\documentclass{amsart}

\usepackage{float}
\usepackage{amsmath}
\usepackage{amsfonts}
\usepackage{amssymb}
\usepackage{graphicx}
\usepackage{hyperref}
\usepackage{mathrsfs}
\usepackage{pb-diagram}
\usepackage{epstopdf}
\usepackage{amsmath,amsfonts,amsthm,enumerate,amscd,latexsym,curves}
\usepackage{bbm}
\usepackage[mathscr]{eucal}
\usepackage{epsfig,epsf,xypic,epic}
\usepackage{caption}
\usepackage{subcaption}

\newtheorem{theorem}{Theorem}[section]

\newtheorem{lemma}[theorem]{Lemma}
\newtheorem{corollary}[theorem]{Corollary}
\newtheorem{question}[theorem]{Question}

\newtheorem{definition}[theorem]{Definition}
\newtheorem{assumption}[theorem]{Assumption}

\newtheorem{example}[theorem]{Example}

\newtheorem{proposition}[theorem]{Proposition}

\newtheorem{remark}[theorem]{Remark}

\newcommand{\dbar}{\bar{\partial}}

\newcommand{\pd}[2]{\frac{\partial #1}{\partial #2}}

\newcommand{\inner}[1]{\langle #1\rangle}

\newcommand{\bb}[1]{\mathbb{#1}}
\newcommand{\cu}[1]{\mathcal{#1}}
\newcommand{\til}[1]{\widetilde{#1}}
\newcommand{\bm}[4]{\left(
\begin{array}{c|c}
#1 & #2 \\
\hline
#3 & #4
\end{array}
\right)}

\begin{document}

\title[SYZ transforms for immersed Lagrangian multi-sections]{SYZ transforms for immersed\\ Lagrangian multi-sections}
\author[K. Chan]{Kwokwai Chan}
\address{Department of Mathematics\\ The Chinese University of Hong Kong\\ Shatin\\ Hong Kong}
\email{kwchan@math.cuhk.edu.hk}
\author[Y.-H. Suen]{Yat-Hin Suen}
\address{Department of Mathematics\\ The Chinese University of Hong Kong\\ Shatin\\ Hong Kong}
\address{\textit{Current address: Center for Geometry and Physics\\ Institute for Basic Science (IBS)\\ Pohang 37673\\ Republic of Korea}}
\email{yhsuen@ibs.re.kr}

\date{\today}

\begin{abstract}
In this paper, we study the geometry of the SYZ transform on a semi-flat Lagrangian torus fibration. Our starting point is an investigation on the relation between Lagrangian surgery of a pair of straight lines in a symplectic 2-torus and extension of holomorphic vector bundles over the mirror elliptic curve, via the SYZ transform for immersed Lagrangian multi-sections defined in \cite{AP, LYZ}. This study leads us to a new notion of equivalence between objects in the immersed Fukaya category of a general compact symplectic manifold $(M, \omega)$, under which the immersed Floer cohomology is invariant; in particular, this provides an answer to a question of Akaho-Joyce \cite[Question 13.15]{AJ}. Furthermore, if $M$ admits a Lagrangian torus fibration over an integral affine manifold, we prove, under some additional assumptions, that this new equivalence is mirror to isomorphism between holomorphic vector bundles over the dual torus fibration via the SYZ transform.
\end{abstract}

\maketitle

\tableofcontents

\section{Introduction}

Mirror symmetry was discovered by string theorists around 1990. It first caught the attention of the mathematical community when Candelas, de la Ossa, Green and Parkes \cite{COGP91} showed that mirror symmetry could be used to predict the number of rational curves in a quintic Calabi-Yau 3-fold. This mysterious phenomenon has continued to attract the attention of numerous mathematicians.

Mirror symmetry is a duality between the symplectic geometry and the complex geometry of two different Calabi-Yau manifolds, which form a so-called {\em mirror pair}. The first mathematical approach towards understanding mirror symmetry was due to Kontsevich \cite{HMS} in 1994. He suggested that mirror symmetry could be phrased as an equivalence between two triangulated categories, namely, the derived Fukaya category on the symplectic side and the derived category of coherent sheaves on the complex side; this is known as the {\em homological mirror symmetry (HMS) conjecture}.

Two years later, Strominger, Yau and Zaslow proposed an entirely geometric approach to explain mirror symmetry, which is now known as the {\em SYZ conjecture} \cite{SYZ}. Roughly speaking, the SYZ conjecture states that mirror symmetry can be understood as a fiberwise duality between two special Lagrangian torus fibrations; moreover, symplectic-geometric (resp. complex-geometric) data on one side can be transformed to complex-geometric (resp. symplectic-geometric) data on the mirror side by a fiberwise Fourier--Mukai-type transform, which we call the {\em SYZ transform}.

The SYZ transform has been constructed and applied to understand mirror symmetry in the semi-flat case \cite{AP, LYZ, Leung05, sections_line_bundles} and the toric case \cite{Abouzaid06, Abouzaid09, Fang08, FLTZ11b, FLTZ12, Chan09, Chan-Leung10a, Chan-Leung12, Chan12, Chan-Ueda12, CPU13, CPU14, Fang16}. But in all of these works the primary focus was on Lagrangian sections and the mirror holomorphic line bundles the SYZ program produces. Applications of the SYZ transform for Lagrangian multi-sections, which should produce higher rank holomorphic vector bundles over the mirror, is largely unexplored.

In this paper we study the geometry of the SYZ transform on a semi-flat Lagrangian torus fibration, focusing on immersed Lagrangian multisections. Construction of the semi-flat SYZ transform will be reviewed in Section \ref{sec:review_semi-flatSYZ}.

In view of the HMS conjecture, Fukaya \cite{Fukaya_multivalued_theta_function}, Seidel and Thomas \cite{Thomas01}, among others, have suggested that Lagrangian surgeries between (graded) Lagrangian submanifolds should be mirror dual to extensions between coherent sheaves over the mirror side. We refer to this as the {\em surgery-extension correspondence}. In Section \ref{sec:surgery_extension}, we investigate this correspondence for the simplest nontrivial example, namely, the 2-torus $T^2$. We will equip the Lagrangian submanifolds with $U(1)$-local systems, which will play a key role in the proof of our correspondence theorem.

More precisely, we consider two Lagrangian straight lines
\begin{align*}
\bb{L}_1 & := \bb{L}_{r_1,d_1}[c_1] := \{(e^{2\pi ir_1x}, e^{2\pi i(d_1x+c_1)})\in T^2: x\in\bb{R}\},\\
\bb{L}_2 & := \bb{L}_{r_2,d_2}[c_2] := \{(e^{2\pi ir_2x}, e^{2\pi i(d_2x+c_2)})\in T^2: x\in\bb{R}\}
\end{align*}
in $T^2$, which are equipped, respectively, with the $U(1)$-local systems
\begin{align*}
\mathcal{L}_{b_1}: d + 2\pi i\frac{b_1}{r_1}dx,\quad \mathcal{L}_{b_2}: d + 2\pi i\frac{b_2}{r_1}dx,\quad b_1, b_2 \in \bb{R}.
\end{align*}
We write $\bb{L}_{1,b_1} = (\bb{L}_1, \mathcal{L}_{b_1}), \bb{L}_{2,b_2} = (\bb{L}_2, \mathcal{L}_{b_2})$ for the A-branes obtained in this way, and denote their SYZ transforms, which are holomorphic vector bundles over the mirror elliptic curve $\check{X}$, by $\check{\bb{L}}_{1,b_1}, \check{\bb{L}}_{2,b_2}$ respectively.
We prove the following surgery-extension correspondence theorem in Section \ref{sec:surgery_extension}:

\begin{theorem}(=Theorem \ref{thm:surgery_extension})\label{thm:surgery_extension_intro}
Let $r_1,d_1,r_2,d_2$ be integers satisfying $r_1d_2 > r_2d_1$ and the gcd conditions
$\text{gcd}(r_1,d_1) = \text{gcd}(r_2,d_2) = \text{gcd}(r_1 + r_2, d_1 + d_2) = 1$.
Let $K\subset \bb{L}_1\cap \bb{L}_2$ be a set of intersection points of $\bb{L}_1$ and $\bb{L}_2$ such that the (graded) Lagrangian surgery produces an immersed Lagrangian
$$\mathbb{L}_K := \bb{L}_2 \sharp_K \bb{L}_1$$
with connected domain, which we then equip with the $U(1)$-local system
$$\mathcal{L}_{b}:d + 2\pi i\frac{b}{r_1+r_2}dx,\quad b \in \bb{R}.$$
Then the SYZ mirror bundle $\check{\mathbb{L}}_{K,b}$ of the Lagrangian A-brane $(\bb{L}_K,\mathcal{L}_b)$ is an extension of $\check{\bb{L}}_{1,b_1}$ by $\check{\bb{L}}_{2,b_2}$, i.e., we have a short exact sequence:
$$0 \to \check{\bb{L}}_{2,b_2} \to \check{\mathbb{L}}_{K,b} \to \check{\bb{L}}_{1,b_1} \to 0$$
if and only if $b$ satisfies the integrality condition
$$b_1 + b_2 - b - \frac{1}{2}\in\bb{Z}.$$
\end{theorem}
In particular, this theorem implies the intriguing phenomenon that the surgery-extension correspondence cannot hold unless we equip Lagrangian submanifolds with suitable {\em nontrivial} local systems (even in the case when we equip $L_1, L_2$ with trivial local systems). 

In Floer-theoretic terms, the integrality condition
in Theorem \ref{thm:surgery_extension_intro} can be regarded as a generalization of degree $-1$ marked points in Abouzaid's work. More precisely, in \cite{Abouzaid_Fukaya_categories_of_higher_genus}, Abouzaid considered immersed curves in Riemann surfaces with one marked point of prescribed degree $-1$, and proved that mapping cones in the Fukaya category can be geometrically realized as Lagrangian surgeries. One may think of the prescribed $-1$ degree for a marked point as the holonomy of a flat $U(1)$-connection concentrated at that point. Our integrality condition recovers Abouzaid's condition by taking $b_1=b_2=b=\frac{1}{2}$.

\begin{remark}
Our theorem is a generalization of a recent result of K. Kobayashi \cite{Ko} to {\em any} rank and degree that satisfy the gcd assumptions.
\end{remark}

\begin{remark}
We believe that the above theorem is known to experts; see in particular \cite[Section 6]{Thomas01}.
\end{remark}

\begin{figure}[H]
\centering
\includegraphics[width=70mm]{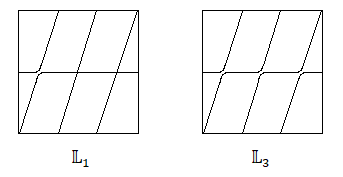}
\caption{Two non-Hamiltonian equivalent immersed Lagrangian multi-sections in $T^2$.}
\label{fig:3_intro}
\end{figure}

By our surgery-extension correspondence theorem, we observe that as long as the two sets of intersection points $K, K'\subset L_1\cap L_2$ are chosen so that the surgeries $\bb{L}_K$ and $\bb{L}_{K'}$ satisfy the assumptions in Theorem \ref{thm:surgery_extension_intro}, their SYZ mirror bundles $\check{\bb{L}}_{K,b}$ and $\check{\bb{L}}_{K',b}$ are isomorphic as holomorphic vector bundles. This is because both bundles are indecomposable and they share the same determinant line bundle $\det(\check{\bb{L}}_{1,b_1})\otimes\det(\check{\bb{L}}_{2,b_2})$, so they must be isomorphic in view of Atiyah's classification of indecomposable vector bundles over elliptic curves \cite{Atiyah_vector_bundle_over_an_elliptic_curve}.
For example, Figure \ref{fig:3_intro} shows two immersed Lagrangian multi-sections $\bb{L}_1, \bb{L}_3$ in $T^2$ which share the same SYZ mirror bundles.

A natural question is then:
\begin{question}\label{ques:sympl_relation}
What is the symplecto-geometric relation between $\bb{L}_K$ and $\bb{L}_{K'}$?
\end{question}

First of all, the relation cannot be the ordinary Hamiltonian equivalence because $\bb{L}_K$ and $\bb{L}_{K'}$ may have a different number of self-intersection points (like in the above example). A na\"ive guess is a weaker notion, called {\em local Hamiltonian equivalence} (Definition \ref{def:loc_Ham}). However, Akaho and Joyce \cite{AJ} pointed out, in view of the {\em Lagrangian $h$-principle} \cite{Gromov_Lagrangian_h, Lee_Lagrangian_h}, local Hamiltonian equivalence is only a weak homotopical notion. It cannot detect `quantum' information, and is therefore too coarse for the immersed Floer cohomology to be invariant.
On the other hand, since the SYZ mirror bundles of $\bb{L}_K$ and $\bb{L}_{K'}$ are isomorphic, the Floer cohomology of $\bb{L}_K$ and $\bb{L}_{K'}$ should also be isomorphic in view of HMS.

This leads us to digress away from SYZ mirror symmetry to study the invariance property of immersed Floer cohomology in Section \ref{sec:inv_thm}, in which we introduce a new equivalence relation on immersed Lagrangian submanifolds called {\em lifted Hamiltonian equivalence}.
\begin{definition}(=Definition \ref{def:M_lifted_Ham})
Let $\pi:\til{M}\to M$ be a finite unramified covering of a symplectic manifold $(M,\omega)$. For two Lagrangian immersions $\mathbb{L}_1 = (L_1, \xi_1), \mathbb{L}_2 = (L_2, \xi_2)$ of $M$, we say $\mathbb{L}_1$ is {\em $(\til{M},\pi)$-lifted Hamiltonian isotopic} to $\mathbb{L}_2$ if there exists a diffeomorphism $\phi:L_1\to L_2$ and Lagrangian immersions $\til{\xi}_1:L_1\to\til{M}$, $\til{\xi}_2:L_2\to\til{M}$ such that $\xi_1=\pi\circ\til{\xi}_1$, $\xi_2=\pi\circ\til{\xi}_2$ and $(L_1,\til{\xi}_1)$ is globally Hamiltonian isotopic (see Definition \ref{def:global_Ham} or \cite[Definition 13.14]{AJ}) to $(L_1,\til{\xi}_2\circ\phi)$ in $(\til{M},\pi^*\omega)$.
\end{definition}

We also make the following
\begin{definition}(=Definition \ref{def:lifted_Ham})
Let $(M,\omega)$ be a symplectic manifold. For two Lagrangian immersions $\mathbb{L}_1 = (L_1, \xi_1), \mathbb{L}_2 = (L_2, \xi_2)$ of $M$, we say $\mathbb{L}_1$ is {\em lifted Hamiltonian isotopic} to $\mathbb{L}_2$ if there exists an integer $l>0$ and Lagrangian immersions $\bb{L}^{(1)}:=\bb{L}_1,\bb{L}^{(2)},\dots,\bb{L}^{(l-1)},\bb{L}^{(l)}:=\bb{L}_2$ of $M$, such that $\bb{L}^{(j)}$ is $(\til{M}_j,\pi_j)$-lifted Hamiltonian isotopic to $\bb{L}^{(j+1)}$, for some finite unramified covering $\pi_j:\til{M}_j\to M$, $j=1,\dots,l-1$.
\end{definition}

This new notion of equivalence is weaker than the usual Hamiltonian equivalence but stronger than local Hamiltonian equivalence (as proved in Corollary \ref{cor:property_lifted_Ham}). In Section \ref{sec:inv_thm}, the following invariance property of immersed Floer cohomology under lifted Hamiltonian equivalences is proved:
\begin{theorem}(=Theorem \ref{thm:inv_thm})
Let $\bb{L}_1, \bb{L}_2$ be Lagrangian immersions in $(M,\omega)$. The Floer cohomology $HF(\bb{L}_1,\bb{L}_2)$ is invariant under lifted Hamiltonian isotopy, i.e., if $\bb{L}_2$ is lifted Hamiltonian isotopic to $\bb{L}_2'$, then there is a quasi-isomorphism
$$(CF(\mathbb{L}_1,\mathbb{L}_2),m_1)\simeq(CF(\mathbb{L}_1,\mathbb{L}_2'),m_1).$$
\end{theorem}
In particular, this gives an answer to a question of Akaho and Joyce \cite[Question 13.15]{AJ}, asking for restricted classes of local Hamiltonian equivalences under which the immersed Lagrangian Floer cohomology is invariant.

In the final Section \ref{sec:mirror_analog_isom}, we go back to SYZ mirror symmetry and Question \ref{ques:sympl_relation}; in fact, we would like to ask an even more general question:
\begin{question}\label{thm:analog}
Let $X\to B$ be a Lagrangian torus fibration and $\check{X}\to B$ be the dual torus fibration. What is the mirror analog of isomorphism between holomorphic vector bundles over $\check{X}$?
\end{question}

We prove that, under certain conditions, the answer is, again, given by lifted Hamiltonian equivalence:
\begin{theorem}(=Theorem \ref{thm:isom_SYZ_mirror})
Suppose that $B$ is compact. Let $\mathbb{L}_1, \mathbb{L}_2$ be immersed Lagrangian multi-sections of $X \to B$ with the same connected domain $L$ and unramified covering map $c_r:L\to B$. Assume that the group of deck transformations $\text{Deck}(L/B)$ acts transitively on fibers of $c_r:L\to B$. Then $\mathbb{L}_1$ is $(L\times_BX,\pi_X)$-lifted Hamiltonian isotopic to $\mathbb{L}_2$ if and only if their SYZ mirrors $\check{\mathbb{L}}_1$ and $\check{\mathbb{L}}_2$ are isomorphic as holomorphic vector bundles over $\check{X}$.
\end{theorem}

Combining this with Theorem \ref{thm:surgery_extension}, we obtain an answer to the earlier Question \ref{ques:sympl_relation}:
\begin{corollary}(=Corollary \ref{cor:torus})
Let $\bb{L}_1 = \bb{L}_{r_1,d_1}[c_1]$ and $\bb{L}_2 = \bb{L}_{r_2,d_2}[c_2]$ be as in Theorem \ref{thm:surgery_extension_intro}. If $K, K'\subset L_1\cap L_2$ are sets of intersection points such that the Lagrangian surgeries $\bb{L}_K= \bb{L}_2\sharp_{K}\bb{L}_1$ and $\bb{L}_{K'}= \bb{L}_2\sharp_{K'}\bb{L}_1$ have connected domain and satisfy the gcd assumption $\text{gcd}(r_1 + r_2, d_1 + d_2) = 1$, then $\bb{L}_K$ and $\bb{L}_{K'}$ are $(S^1\times_{S^1}T^2,\pi_{T^2})$-lifted Hamiltonian equivalent, and hence have isomorphic immersed Lagrangian Floer cohomologies.
\end{corollary}

\section*{Acknowledgment}
We are grateful to Hanwool Bae, Cheol-Hyun Cho, Hansol Hong, Wonbo Jeong, Conan Leung and Cheuk Yu Mak for various useful discussions and for providing us with insightful comments and suggestions. We would like to thank Professor Shing-Tung Yau for his encouragement and interest in our work. Thanks are also due to the anonymous referees for carefully reading an earlier version of this paper, pointing out various errors and giving us a long list of constructive comments and suggestions.

The work of K. Chan described in this paper was substantially supported by grants from the Research Grants Council of the Hong Kong Special Administrative Region, China (Project No. CUHK14302015 $\&$ CUHK14302617). The work of Y.-H. Suen was supported by IBS-R003-D1.

\section{Semi-flat mirror symmetry}\label{sec:review_semi-flatSYZ}

In this section, we review the SYZ transform in the semi-flat setting, following \cite{AP} and \cite{LYZ} (see also \cite{Leung05}, \cite[Section 2]{Chan-Leung10b} or \cite[Section 2]{Chan12}).

\subsection{SYZ mirror construction}

Let $B$ be an $n$-dimensional integral affine manifold, meaning that the transition functions of $B$ belong to the group $\bb{R}^n\rtimes GL(n,\bb{Z})$ of $\bb{Z}$-affine linear maps. Let $\Lambda\subset TB$ and $\Lambda^*\subset T^*B$ be the natural lattice bundles defined by the integral affine structure. More precisely, on a local affine chart $U\subset B$, we define
$$\Lambda(U):=\bigoplus_{j=1}^n\bb{Z}\cdot\pd{}{x^j},\quad \Lambda^*(U):=\bigoplus_{j=1}^n\bb{Z}\cdot dx^j,$$
where $(x^j)$ are affine coordinates of $U$.

We set
$$X := T^*B/\Lambda^*\text{  and  }\check{X} := TB/\Lambda$$
and let $(y^j), (\check{y}_j)$ be fiber coordinates (which are dual to each other) of $X$ and $\check{X}$ respectively. Then $(x^j,y^j)$ and $(x^j,\check{y}_j)$ define a set of local coordinates on $T^*U/\Lambda^*\subset X$ and $TU/\Lambda\subset\check{X}$ respectively. We also let $\pi_B:X\to B$ and $\check{\pi}_B:\check{X}\to B$ be the natural projections.

Equip $X$ with the standard symplectic structure
$$\omega_{\hbar}:=\hbar^{-1}\sum_jdy^j\wedge dx^j,$$
where $\hbar>0$ is a small real parameter. This defines a family of symplectic manifolds $(X,\omega_{\hbar})$. As $\hbar\to 0$, the symplectic volume of $(X,\omega_{\hbar})$ approaches infinity, which is the so-called {\em large volume limit} of the family $\{(X,\omega_{\hbar})\}_{\hbar>0}$.

On the other hand, there is a natural almost complex structure $\check{J}_{\hbar}$ on $\check{X}$ given by
$$\check{J}_{\hbar}\left(\pd{}{x^j}\right)=-\hbar^{-1}\pd{}{\check{y}_j}\text{  and  }\check{J}_{\hbar}\left(\pd{}{\check{y}_j}\right)=\hbar\pd{}{x^j}.$$
It is easy to see that $\check{J}_{\hbar}$ is indeed integrable with local complex coordinates given by $z_j=\check{y}_j+ix^j$. Hence $(\check{X},\check{J}_{\hbar})$ defines a family of complex manifolds approaching the so-called {\em large complex structure limit} as $\hbar \to 0$.

\begin{definition}
$(\check{X},\check{J}_{\hbar})$ is called the {\em SYZ mirror} of $(X,\omega_{\hbar})$.
\end{definition}

As the above limiting processes do not play any role in this paper, by absorbing $\hbar^{-1}$ into the $(x^j)$-coordinates, we will simply assume that $\hbar=1$ throughout this paper. Hence we just write $\omega$ for $\omega_{\hbar}$ and $\check{J}$ for $\check{J}_{\hbar}$.

\subsection{The SYZ transform of branes}

In order for homological mirror symmetry to make sense, one needs to complexify the Fukaya category by equipping Lagrangian submanifolds with unitary local systems \cite{HMS}. Here, we just consider rank 1 local systems on Lagrangian submanifolds.

\begin{definition}\label{def:Lag_imm}
A {\em Lagrangian immersion} $\bb{L}$ of $(X,\omega)$ is a pair $(L,\xi)$, where $L$ is an $n$-dimensional smooth manifold and $\xi:L\to X$ is an immersion with the following properties
\begin{itemize}
\item [a)] $\xi^*\omega=0$.
\item [b)] There is a discrete set of points $S\subset L$ such that $\xi:L\backslash S\to X$ is injective.
\item [c)] For all $p\in X$, the set $\xi^{-1}(p)\cap S\subset L$ is either empty or consists of two points.
\end{itemize}
An {\em A-brane} of $X$ is a pair $(\bb{L},\mathcal{L})$, where $\bb{L}$ is an immersed Lagrangian submanifold of $X$ and $\mathcal{L}$ is a rank 1 unitary local system on $L$.
\end{definition}

We shall focus on the case where $\bb{L}$ is an immersed Lagrangian multi-section of the fibration $\pi_B:X\to B$.

\begin{definition}
An {\em immersed Lagrangian multi-section} of rank $r$ is a triple $\bb{L} := (L, \xi, c_r)$, where $\xi:L\to X$ is a Lagrangian immersion and $c_r:L\to B$ is an $r$-fold unramified covering map such that $\pi_B\circ\xi=c_r$. We also assume that the image of $L$ intersects transversally with each torus fiber.
\end{definition}

\begin{remark}
We remark that $L$ is not necessarily connected.
\end{remark}

We now define the SYZ transform of an immersed Lagrangian multi-section in a semi-flat Lagrangian torus fibration, following \cite{AP, LYZ} (see also \cite{BMP_I, BMP_II} for a very similar exposition).

Let $\mathcal{P}\to X\times_B\check{X}$ be the Poincar\'e line bundle, whose total space is defined as the quotient
$$\cu{P}:= \left( (T^*B\oplus TB)\times\bb{C} \right) / \left(\Lambda^*\oplus\Lambda\right),$$
where the fiberwise action of $\Lambda^*\oplus\Lambda$ on $(T^*B\oplus TB)\times\bb{C}$ is given by
$$(\lambda,\check{\lambda})\cdot(y,\check{y},t) := \left( y+\lambda,\check{y}+\check{\lambda},e^{i\pi(\inner{y,\check{\lambda}}-\inner{\lambda,\check{y}})}\cdot t \right).$$
Define a connection $\nabla_{\mathcal{P}}$ on $\mathcal{P}$ by
$$\nabla_{\mathcal{P}}:=d+i\pi(\inner{y,d\check{y}}-\inner{\check{y},dy}).$$
The section $e^{i\pi(y,\check{y})}$ is invariant under the $\{0\}\oplus\Lambda$ action:
$$(0,\check{\lambda})\cdot(y,\check{y},t)=\left(y,\check{y}+\check{\lambda},e^{i\pi\inner{y,\check{\lambda}+\check{y}}}\right).$$
Hence it descends to a section on $T^*B\times_B\check{X}$. With respect to this frame, the connection $\nabla_{\mathcal{P}}$ can be written as
$$\nabla_{\mathcal{P}}=d+2\pi i\inner{y,d\check{y}}.$$
The remaining action of $\Lambda^*\oplus\{0\}$ then becomes
\begin{align*}
\lambda\cdot[(y,\check{y},e^{i\pi\inner{y,\check{y}}})]_{{\Lambda}}
& = [y+\lambda,\check{y},e^{-i\pi\inner{\lambda,\check{y}}}\cdot e^{i\pi\inner{y,\check{y}}}]_{{\Lambda}}\\
& = e^{-2\pi i\inner{\lambda,\check{y}}}[y+\lambda,\check{y},e^{i\pi\inner{y+\lambda,\check{y}}}]_{\Lambda}.
\end{align*}

Let $\bb{L}=(L,\xi,c_r)$ be an immersed Lagrangian multi-section of rank $r$ and $\cu{L}$ be a $U(1)$-local system on $L$. Define
$$\check{\bb{L}}:=(\pi_{\check{X}})_*\left((\xi\times id_{\check{X}})^*(\cu{P})\otimes(\pi_L^*\cu{L})\right).$$
Note that as the projection map $\pi_{\check{X}}:L\times_B\check{X}\to\check{X}$ is an unramified $r$-fold covering map, $\check{\bb{L}}$ is a vector bundle of rank $r$. The connection on $\cu{P}$ induces a natural connection $\nabla_{\check{\bb{L}}}$ on $\check{\bb{L}}$. The following proposition is standard (see the original papers \cite{AP, LYZ} or \cite[Section 2]{Chan-Leung10b}, \cite[Section 2]{Chan12}):

\begin{proposition}
The connection $\nabla_{\check{\bb{L}}}$ satisfies $(\nabla_{\check{\bb{L}}}^2)^{0,2}=0$ if and only if the immersion $\xi:L\to X$ is Lagrangian.
\end{proposition}

Hence $\check{\bb{L}}$ carries a natural holomorphic structure.

\begin{definition}
$(\check{\bb{L}},\nabla_{\check{\bb{L}}})$ is called the {\em SYZ mirror bundle} of the A-brane $(\bb{L},\cu{L})$. We simply write $\check{\bb{L}}$ for short.
\end{definition}

Let us give a more detailed local description of $\check{\bb{L}}$ and $\nabla_{\check{\bb{L}}}$ for the case $r=1$.

We first suppose that $\cu{L}\to L$ is the trivial line bundle equipped with the trivial connection. Let $U$ be an affine chart of $B$. Take a lift of $L\cap T^*U/\Lambda^*\subset X$ to $\widetilde{L}_U\subset T^*U$, and let $\xi_U$ be the defining equation of $\widetilde{L}_U$. The section $e^{i\pi(\xi_U,\check{y})}$ on $\til{L}_U\times_BTU$ induces a section $\check{1}_U$ of $\check{\bb{L}}$ on $\til{L}_U\times_B\check{X}$ by taking its $\Lambda$-equivalence class. With respect to this local frame, the connection $\nabla_{\check{\bb{L}}}$ becomes
$$\nabla_{\check{\bb{L}}}=d+2\pi i\inner{\xi_U,d\check{y}}.$$

We can also compute the unitary and holomorphic transition functions of $(\check{\bb{L}},\nabla_{\check{\bb{L}}})$.
Let $V$ be another affine chart of $B$ such that $U\cap V\neq\phi$. Let $\xi_V$ be the defining equation of the lift $\widetilde{L}_V\subset T^*V$. Since
$$[\inner{\xi_U,dx_U}]_{\Lambda^*}=[\inner{\xi_V,dx_V}]_{\Lambda^*},$$
there exists $\lambda_{UV}\in\Lambda^*|_{U\cap V}$ such that
$$\inner{\xi_U,dx_U}=\inner{\xi_V,dx_V}+\inner{\lambda_{UV},dx_V}.$$
Then we have
$$\check{1}_V=e^{-2\pi i\inner{\lambda_{UV},\check{y}_V}}\check{1}_U.$$
Therefore the unitary transition functions are given by
\begin{equation}\label{eqn:unitary-transition_fcs}
\tau_{UV}(x_V)=e^{2\pi i\inner{\lambda_{UV},\check{y}_V}}.
\end{equation}

To compute the holomorphic one, let $f_U:U\to\mathbb{R}$ be a primitive of $\xi_U$. Then it is easy to check that
\begin{equation}\label{eqn:local_holomorphic frame}
\check{e}_U=e^{-2\pi f_U}\check{1}_U
\end{equation}
defines a local holomorphic frame of $\check{\bb{L}}$. Since $f_U,f_V$ are primitives of $\inner{\xi_U,dx_U},\inner{\xi_V,dx_V}$ respectively, we have
$$f_U(x_U)=f_V(x_V)+\inner{\lambda_{UV},x_V}+c_{UV}$$
for some $c_{UV}\in\bb{R}$. The holomorphic transition functions are then given by
\begin{equation}\label{eqn:holom_transition_fcs}
g_{UV}(z_V)=e^{-2\pi c_{UV}}e^{2\pi i\inner{\lambda_{UV},z_V}},
\end{equation}
where $z_V=\check{y}_V+ix_V$ is a holomorphic coordinate of $TV/\Lambda$.

Now, suppose $\cu{L}\to L$ is an arbitrary $U(1)$-local system. Then $\cu{L}$ carries a natural flat connection $\nabla_{\cu{L}}$. Write
$$\nabla_{\mathcal{L}}=d+2\pi i\beta,\quad\beta\in\Gamma(L,T^*L).$$
Since $\nabla_{\mathcal{L}}^2=0$, $d\beta=0$. Let $b_U(x_U)$ be a primitive of $\beta$ on $U$. Then a local holomorphic frame is given by
$$e^{-2\pi(f_U+ib_U)}\check{1}_U.$$
When $U\cap V\neq\phi$, we have
$$db_U(x_U)=\beta=db_V(x_V)\Rightarrow b_U(x_U)-b_V(x_V)=b_{UV}\in\bb{R}.$$
The holomorphic transition functions then become
$$g_{UV}(z_V)=e^{-2\pi(c_{UV}+ib_{UV})}e^{2\pi i(\lambda_{UV},z_V)}.$$
The connection $\nabla_{\check{\bb{L}}}$ becomes
$$\nabla_{\check{\bb{L}}}=d+2\pi i\inner{\xi_U,d\check{y}}+2\pi i\beta|_U.$$

For a general $r$, one can also write down the connection $\nabla_{\check{\mathbb{L}}}$ in terms of the data coming from the Lagrangian brane:
\begin{equation}\label{eqn:rank_r_connection}
\begin{split}
\nabla_{\check{\mathbb{L}}} = & d + 2\pi i\sum_{j=1}^n
\begin{pmatrix}
\xi_{U_{1,j}}(x_1) & 0 & 0 & \dots  & 0 \\
0 & \xi_{U_{2,j}}(x_2) & 0 & \dots  & 0 \\
\vdots & \vdots & \vdots & \ddots & \vdots \\
0 & 0 & 0 & \dots  & \xi_{U_{r,j}}(x_r)
\end{pmatrix}d\check{y}_U^j\\
& \qquad + 2\pi i\sum_{j=1}^n
\begin{pmatrix}
\beta_{U_{1,j}}(x_1) & 0 & 0 & \dots  & 0 \\
0 & \beta_{U_{2,j}}(x_2) & 0 & \dots  & 0 \\
\vdots & \vdots & \vdots & \ddots & \vdots \\
0 & 0 & 0 & \dots  & \beta_{U_{r,j}}(x_r)
\end{pmatrix}d\check{x}_U^j,
\end{split}
\end{equation}
where $c_r^{-1}(U)=\coprod_{k=1}^rU_k$ and $x_k\in U_k$, $k=1,\dots,r$ are the preimages of $x\in U\subset B$.

\section{Surgery-extension correspondence for $T^2$}\label{sec:surgery_extension}

Let $\bb{L}_1$ and $\bb{L}_2$ be two graded immersed Lagrangian multi-sections and $\check{\bb{L}}_1$ and $\check{\bb{L}}_2$ be their mirror bundles. It is believed that performing Lagrangian surgeries at index 1 intersection points of $\bb{L}_1$ and $\bb{L}_2$ corresponds to forming a nontrivial extension of $\check{\bb{L}}_1$ and $\check{\bb{L}}_2$. More precisely, let $K:=\{p_1,\cdots,p_k\}\subset CF(\bb{L}_1,\bb{L}_2)$ be a collection of index 1 intersection points of $\bb{L}_1$ and $\bb{L}_2$. We perform Lagrangian surgery at each point in $K$ (see Figure \ref{fig:Lag_surgery}) to obtain another graded immersed Lagrangian multi-section $\mathbb{L}_K:=\bb{L}_2\sharp_K \bb{L}_1$. Then the mirror bundle $\check{\mathbb{L}}_K$ of $\mathbb{L}_K$ should fit in an exact sequence:
$$0\to\check{\bb{L}}_2\to\check{\mathbb{L}}_K\to\check{\bb{L}}_1\to 0.$$

In this section, we study this relation on the symplectic torus $T^2$ with standard symplectic structure and its mirror elliptic curve. We will see that the Lagrangian surgery and extension correspondence cannot be true in general if we do not equip the Lagrangians with $U(1)$-local systems.

\begin{figure}[H]
\center
\includegraphics[width=105mm]{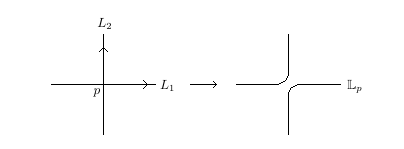}
\caption{1-dimensional Lagrangian surgery at an index 1 intersection point.}
\label{fig:Lag_surgery}
\end{figure}

Let $X:=T^2=S^1\times S^1$ be the product torus with standard symplectic structure given by
$$\omega:=dy\wedge dx.$$
We begin with describing the SYZ transform of a general immersed Lagrangian multi-section in $X$.

Let $B=S^1$ and $\pi_1:X\to S^1$ be the projection onto the first factor. Let $\varphi:\mathbb{R}\to\mathbb{R}$ be a smooth function such that
$$\varphi(x+r)=\varphi(x)+d,$$
where $d \in \mathbb{Z}$ and $r \in \mathbb{Z}_{>0}$.
Then $\varphi$ descends to an immersed Lagrangian multi-section $\bb{L}_{\varphi}$ of $\pi_1:X\to B$ which intersect the zero section $|d|$ times and each fiber $r$ times. Since $\varphi$ is smooth, the immersed Lagrangian multi-section $\bb{L}_{\varphi}$ intersects the fibers of $\pi_1:X\to B$ transversally. Clearly, every immersed Lagrangian multi-section with connected domain and which intersects the fibers transversally arises in this manner.

Let $\{U,V\}$ be the following affine cover of the base $B=S^1$:
\begin{align*}
(0,1) & \to V \subset S^1,\quad x\mapsto e^{2\pi ix},\\
(0,1) & \to U \subset S^1,\quad x'\mapsto e^{2\pi i(x'+\epsilon)},
\end{align*}
where $\epsilon\in(0,1)$ is fixed. Write $U\cap V=W_1\amalg W_2$. The SYZ mirror bundle $\check{\bb{L}}_{\varphi}$ of $\bb{L}_{\varphi}$ is a rank $r$ vector bundle with $U(r)$-connection (cf. \eqref{eqn:rank_r_connection})
$$
\nabla_{\check{\mathbb{L}}_{\varphi}}=d+2\pi i
\begin{pmatrix}
\varphi(x) & 0 & 0 & \dots  & 0 \\
0 & \varphi(x+1) & 0 & \dots  & 0 \\
\vdots & \vdots & \vdots & \ddots & \vdots \\
0 & 0 & 0 & \dots  & \varphi(x+r-1)
\end{pmatrix}
d\check{y}, \quad (x,\check{y})\in TV/\Lambda.
$$
With the complex structure $z=\check{y}+ix$, the degree of $\check{\bb{L}}_{\varphi}$ is given by
\begin{align*}
\frac{i}{2\pi}\int_0^1\int_0^1\frac{d}{dx}\left(2\pi i\sum_{j=0}^{r-1}\varphi(x+j)\right)dx\wedge d\check{y}
& = -\int_0^r\frac{d}{dx}\varphi(x)dx\\
& = \varphi(0)-\varphi(r) = -d.
\end{align*}

Let us write down the unitary and holomorphic transition functions of $\check{\bb{L}}_{\varphi}$ from $TV/\Lambda$ to $TU/\Lambda$. Let
$$c_r^{-1}(U) = \coprod_{j=1}^r U_j \text{ and } c_r^{-1}(V) = \coprod_{j=1}^r V_j.$$
By renaming, we can assume that $U_1\cap V_1,V_1\cap U_2,U_2\cap V_2,\dots,V_r\cap U_1$ are non-empty connected subsets of $L\times_B\check{X}$.
Then we can assume that
$$c_r^{-1}(W_1) = \coprod_{j=1}^r U_j\cap V_j \text{ and } c_r^{-1}(W_2) = \coprod_{j=1}^r V_j\cap U_{j+1},$$
where we put $U_{r+1}=U_1$.
The unitary transition function is given by the identity matrix on $TW_1/\Lambda$ and by
$$\tau_{UV}(x,\check{y}) = \bm{O}{e^{2\pi id\check{y}}}{I_{(r-1)\times(r-1)}}{O}$$
on $TW_2/\Lambda$ (cf. \eqref{eqn:unitary-transition_fcs}).

Choose a point $x_0\in W_1\subset V$. For each $j=0,\dots,r-1$, let $f_j:\bb{R}\to\bb{R}$ be given by
\begin{equation}\label{eqn:f_j}
f_j(x) := \int_{x_0}^{x+j}\varphi(u)du,
\end{equation}
which is a primitive of $\varphi(x+j)$. A local holomorphic frame for $\check{\bb{L}}_{\varphi}$ on the chart $TV/\Lambda$ is then given by (cf. \eqref{eqn:local_holomorphic frame}):
$$\{e^{-2\pi f_0(x)}\check{1}_0(x,\check{y}),\dots,e^{-2\pi f_{r-1}(x)}\check{1}_{r-1}(x,\check{y})\},$$
where $(x,\check{y})\in TV/\Lambda$.

On the $U$-chart, we have
\begin{equation}\label{eqn:f'_j}
f'_j(x')=\int_{x_0'+\epsilon}^{x'+\epsilon+j}\varphi(u)du.
\end{equation}
On $W_1$, $x'$ and $x$ are related by $x'(x)=x-\epsilon$. Applying this coordinate change to \eqref{eqn:f'_j} and using \eqref{eqn:f_j}, we obtain
$$f'_j(x')=\int_{x_0'+\epsilon}^{x'+\epsilon+j}\varphi(u)du=\int_{x_0}^{x+j}\varphi(u)du=f_j(x).$$
On $W_2$, we have $x'(x)=x-\epsilon+1$, so for $j=0,\dots,r-2$, \eqref{eqn:f'_j} becomes
$$f'_j(x')=\int_{x_0}^{x+j+1}\varphi(u)du=f_{j+1}(x),$$
while for $j=r-1$, we have
\begin{equation}\label{eqn:f'_r-1}
f'_{r-1}(x')=\int_{x_0}^{x+r}\varphi(u)du=f_0(x)+\int_{x}^{x+r}\varphi(u)du.
\end{equation}
Since
$$\frac{d}{dx}\int_{x}^{x+r}\varphi(u)du=\varphi(x+r)-\varphi(x)=d,$$
we have
\begin{equation}\label{eqn:int_varphi}
\int_{x}^{x+r}\varphi(u)du=dx+\int_0^r\varphi(u)du.
\end{equation}
By substituting \eqref{eqn:int_varphi} to \eqref{eqn:f'_r-1}, we see that the holomorphic transition functions are given by the identity matrix on $TW_1/\Lambda$ and by
$$g_{UV}(z)=\bm{O}{e^{-2\pi a}e^{2\pi idz}}{I_{(r-1)\times(r-1)}}{O}$$
on $TW_2/\Lambda$ (cf. \eqref{eqn:holom_transition_fcs}), where $a$ is given by
$$a=\int_0^r\varphi(u)du.$$

Suppose now we enrich $\bb{L}$ by a $U(1)$-local system $(\cu{L},\nabla_{\cu{L}})$. Since the domain of $\bb{L}$ is a circle, the connection $\nabla_{\cu{L}}$ can always be written as
$$d+2\pi i\frac{b}{r}dx,$$
for some $b\in\bb{R}$. Hence the transition functions of the SYZ mirror bundle of $(\bb{L},\cu{L})$ are given by
$$g_{UV}(z)=\bm{O}{e^{-2\pi (a+ib)}e^{2\pi idz}}{I_{(r-1)\times(r-1)}}{O}.$$

\begin{example}
Let $r,d\in\bb{Z}$ with $\text{gcd}(r,d)=1$. Let $\varphi:\bb{R}\to\bb{R}$ be the straight line
$$\varphi(x)=\frac{d}{r}x+\frac{c}{r},\quad r>0,c\in\bb{R}.$$
Then it descends to the Lagrangian multi-section
$$\bb{L}_{r,d}[c]=\{(e^{2\pi irx},e^{2\pi i(dx+c)})\in X:x\in\bb{R}\}.$$
One computes that
$$\int_0^r\varphi(x)dx=\frac{rd}{2}+c.$$
The transition function of the SYZ mirror bundle is given by
$$g_{UV}(z)=\bm{O}{e^{-\pi dr}e^{-2\pi c}e^{2\pi idz}}{I_{(r-1)\times(r-1)}}{O}.$$
\end{example}

Let
\begin{align*}
\bb{L}_1 & := \bb{L}_{r_1,d_1}[c_1] := \{(e^{2\pi ir_1x},e^{2\pi i(d_1x+c_1)})\in X:x\in\mathbb{R}\},\\
\bb{L}_2 & := \bb{L}_{r_2,d_2}[c_2] := \{(e^{2\pi ir_2x},e^{2\pi i(d_2x+c_2)})\in X:x\in\mathbb{R}\}
\end{align*}
be two distinct (embedded) Lagrangian multi-sections of $X$, where $r_1,r_2,d_1,d_2$ are integers such that $\text{gcd}(r_1,d_1)=\text{gcd}(r_2,d_2)=1$ and $c_1,c_2\in\bb{R}$. They intersect at $|r_1d_2-r_2d_1|$ points. Let us assume that $r_1d_2>r_2d_1$. The base coordinates of the intersection points are given by the equivalence classes of
$$x_{k,k'}:=\frac{r_2(c_1+k)-r_1(c_2+k')}{r_1d_2-r_2d_1},\quad k,k'\in\mathbb{Z}.$$
We orientate $\bb{L}_1,\bb{L}_2$ such that both of them are pointing towards ``right" in a fundamental domain of $X$. Since $r_1d_2>r_2d_1$, using the degree convention of \cite{Abouzaid_Fukaya_categories_of_higher_genus}, all generators of $CF(\bb{L}_1,\bb{L}_2)$ are of index 1.

Let $K$ be a subset of $\bb{L}_1\cap \bb{L}_2$. Then we can perform Lagrangian surgery at each point in $K$ to obtain a (graded) Lagrangian multi-section $\mathbb{L}_K:=\bb{L}_2\sharp_K\bb{L}_1$ (possibly with disconnected domain).

\begin{remark}
For each surgery point, we have a parameter $\epsilon>0$ which controls the size of the surgery. The surgery $\bb{L}_K$ we discuss here of course consists of the surgery parameters. However, these parameters do not play a role as we will see in the proof of our main theorem (Theorem \ref{thm:surgery_extension}).
\end{remark}

Now, we equip the domains of $\bb{L}_1,\bb{L}_2,\bb{L}_K$ with the $U(1)$-local systems
\begin{align*}
\cu{L}_{b_1} & : d+2\pi i\frac{b_1}{r_1}dx,\\
\cu{L}_{b_2} & : d+2\pi i\frac{b_2}{r_2}dx,\\
\cu{L}_b & : d+2\pi i\frac{b}{r_1+r_2}dx
\end{align*}
respectively. Here $b_1,b_2,b\in\bb{R}$. We denote the Lagrangian A-branes $(\bb{L}_1,\cu{L}_{b_1}),(\bb{L}_2,\cu{L}_{b_2}),\\(\bb{L}_K,\cu{L}_b)$ by $\bb{L}_{1,b_1},\bb{L}_{2,b_2},\bb{L}_{K,b}$ respectively. There is a simple obstruction for the SYZ mirror bundle $\check{\bb{L}}_{K,b}$ of $\bb{L}_{K,b}$ to be an extension of $\check{\bb{L}}_{1,b_1}$ and $\check{\bb{L}}_{2,b_2}$. Note that if $\check{\bb{L}}_{K,b}$ is an extension of $\check{\bb{L}}_{1,b_1}$ by $\check{\bb{L}}_{2,b_2}$, then $\det(\check{\bb{L}}_{K,b})\cong\det(\check{\bb{L}}_{1,b_1})\otimes\det(\check{\bb{L}}_{2,b_2})$ as holomorphic line bundles.

\begin{proposition}\label{prop:same_det}
Let $(\bb{L}_{r_1,d_1}[c_1],\mathcal{L}_1)$ and $(\bb{L}_{r_2,d_2}[c_2],\mathcal{L}_2)$ be Lagrangian A-branes with local systems $$\mathcal{L}_{b_1}:d+2\pi i\frac{b_1}{r_1}dx,\quad \mathcal{L}_{b_2}:d+2\pi i\frac{b_2}{r_2}dx.$$
Let $\bb{L}_1',\dots,\bb{L}_M'$ be the components of an immersed Lagrangian multi-section $\bb{L}$ of rank $r_1+r_2$ and degree $-d_1-d_2$, equipped with $U(1)$-local systems
$$\mathcal{L}_1':d+2\pi i\frac{b_1'}{r_1'}dx,\dots,\mathcal{L}_M':d+2\pi i\frac{b_M'}{r_M'}dx.$$
Let $\varphi_j':\bb{R}\to\bb{R}$ be the defining equation of $\bb{L}_j'$. Put
$$a_j'=\int_{0}^{r_j'}\varphi_j'(x)dx,\quad j=1,\dots, M,$$
Then $\det(\check{\bb{L}}_{b'})\cong\det(\check{\bb{L}}_{r_1,d_1}[c_1]_{b_1})\otimes\det(\check{\bb{L}}_{r_2,d_2}[c_2]_{b_2})$ as holomorphic line bundles if and only if
$$\sum_{j=1}^Ma_j'-\frac{r_1d_1}{2}-\frac{r_2d_2}{2}-c_1-c_2\in\bb{Z}\text{ and  }b_1+b_2-\frac{M}{2}-\sum_{j=1}^Mb_j'\in\bb{Z}.$$
\end{proposition}
\begin{proof}
Note that
$$\sum_{j=1}^Mr_j'=r_1+r_2,\quad\sum_{j=1}^Md_j'=d_1+d_2.$$
Using the above computations of the holomorphic transition functions, we see that the factor of automorphy of $\det(\bb{L})\otimes\det(\check{\bb{L}}_1)^{-1}\otimes\det(\check{\bb{L}}_2)^{-1}$ over $\bb{C}^\times$ (instead of over $\bb{C}$) is generated by
\begin{align*}
A(1,u) & = (-1)^M e^{-2\pi\left( \sum_{j=1}^Ma_j'-\frac{r_1d_1}{2}-\frac{r_2d_2}{2}-c_1-c_2 \right)}e^{2\pi i\left( \sum_{j=1}^Mb_j-b_1-b_2 \right)}\\
& = e^{-2\pi\left( \sum_{j=1}^Ma_j'-\frac{r_1d_1}{2}-\frac{r_2d_2}{2}-c_1-c_2 \right)}e^{2\pi i\left( \sum_{j=1}^M b_j+\frac{M}{2}-b_1-b_2 \right)},
\end{align*}
where $u=e^{2\pi i z}$ (see e.g. Section 4.3 in \cite{factor_of_automorphy} for the precise relation between transition functions and factors of automorphy over $\bb{C}^\times$). 

By Theorem 4.11 in \cite{factor_of_automorphy}, $A(1,u)$ is gauge equivalent to $1$ if and only if
\begin{equation}\label{eqn:automorphy_gauge_equiv}
B(e^{-2\pi}u)=A(1,u)B(u)
\end{equation}
for some non-zero holomorphic function $B:\bb{C}^{\times}\to\bb{C}^{\times}$. Let $\sum_{-\infty}^{\infty}B_ku^k$ be the Laurent series expansion of $B$. Then \eqref{eqn:automorphy_gauge_equiv} holds if and only if
$$B_k(e^{-2\pi k}-A(1,u))=0, \text{ for all }k.$$
Since $B(u)$ is nonzero, \eqref{eqn:automorphy_gauge_equiv} holds if and only if $A(1,u)=e^{-2\pi N}$ for some integer $N$ and $B_k=0$ for all $k\neq N$, which is equivalent to saying that
$$\sum_{j=1}^Ma_j'-\frac{r_1d_1}{2}-\frac{r_2d_2}{2}-c_1-c_2=N\in\bb{Z}\text{ and }b_1+b_2 - \frac{M}{2}-\sum_{j=1}^Mb_j'\in\bb{Z}.$$
\end{proof}

We refer to the conditions
$$\sum_{j=1}^Ma_j'-\frac{r_1d_1}{2}-\frac{r_2d_2}{2}-c_1-c_2\in\bb{Z}\text{ and  }b_1+b_2-\frac{M}{2}-\sum_{j=1}^Mb_j'\in\bb{Z}$$
as the first and second integrality condition for the triple $(\bb{L}_{r_1,d_1}[c_1]_{b_1},\bb{L}_{r_2,d_2}[c_2]_{b_2},\bb{L}_{b'})$ respectively.

Proposition \ref{prop:same_det} gives a necessary condition for the surgery-extension correspondence to hold. Next, we prove that under certain assumptions on the surgery $\bb{L}_K$, the second integrality condition is also sufficient.

\begin{theorem}\label{thm:surgery_extension}
Let $r_1,d_1,r_2,d_2$ be integers satisfying $r_1d_2>r_2d_1$ and
$$\text{gcd}(r_1,d_1)=\text{gcd}(r_2,d_2)=\text{gcd}(r_1+r_2,d_1+d_2)=1.$$
Let $\bb{L}_1:=\bb{L}_{r_1,d_1}[c_1]$, $\bb{L}_2:=\bb{L}_{r_2,d_2}[c_2]$ and $K\subset \bb{L}_1\cap \bb{L}_2$ such that the (graded) Lagrangian surgery $\mathbb{L}_K:=\bb{L}_2\sharp_K\bb{L}_1$ has connected domain and is equipped with the $U(1)$-local system
$$\mathcal{L}_{b}:d+2\pi i\frac{b}{r_1+r_2}dx,\quad b\in\bb{R}.$$
Then the SYZ mirror bundle $\check{\mathbb{L}}_{K,b}$ of the Lagrangian A-brane $\bb{L}_{K,b}$ is an extension of $\check{\bb{L}}_{1,b_1}$ by $\check{\bb{L}}_{2,b_2}$ if and only if $b$ satisfies
$$b_1+b_2+\frac{1}{2}-b\in\bb{Z}.$$
\end{theorem}

To prove this theorem, we need some results on semistable vector bundles on algebraic curves from \cite{exact_sequence_stable_bundle}:

\begin{lemma}[Lemma 1.4 and Proposition 2.3 in \cite{exact_sequence_stable_bundle}]\label{lem:surjective}
Let $F,G$ be polystable vector bundles over an elliptic curve $X$ with $\text{rk}(F)\geq \text{rk}(G)$ and $\mu(F)<\mu(G)$. Assume that no two among the
indecomposable factors of $F$ (resp. of $G$) are isomorphic. Define
\begin{align*}
U & := \{f\in\text{Hom}(F,G): \text{rk}(\text{Im}(f))=t, \deg(\text{Im}(F))=h\},\\
t & := \max_{f\in\text{Hom}(F,G)} \text{rk}(\text{Im}(f)),\\
h & := \max_{f\in\text{Hom}(F,G),\text{rk}(\text{Im}(f))=t} \deg(\text{Im}(F)).
\end{align*}
Then $U$ is an open dense subset of $\text{Hom}(F,G)$. Moreover, if $\text{rk}(F) > \text{rk}(G)$, then each $f\in U$ is surjective.
\end{lemma}

\begin{lemma}[Corollary 1.3 in \cite{exact_sequence_stable_bundle})]\label{lem:flat_family_limit}
Fix a flat family $\{X(t):t\in T\}$ of smooth compact Riemann surfaces with $T$ integral. Let $H,Q$ be vector bundles on $X=X(0)$ such that $\text{Hom}(Q(0),H(0))=0$. Then any extension
$$0\to H(0)\to E\to Q(0)\to 0$$
is the limit of a flat family of extensions
$$0\to H(t)\to E(t)\to Q(t)\to 0$$
with $H(t)$ and $Q(t)$ semi-stable and $t$ in some open subset of $T$ containing $0$.
\end{lemma}

\begin{proof}[Proof of Theorem \ref{thm:surgery_extension}]
Let $\varphi:\mathbb{R}\to\mathbb{R}$ be the defining equation of $\mathbb{L}_K$ such that $\varphi(0)=a_1/r_1$. The first integrality condition:
$$\int_0^{r_1+r_2}\varphi(x)dx=\frac{r_1d_1}{2}+\frac{r_2d_2}{2}+c_1+c_2+N,\text{ for some }N\in\mathbb{Z}$$
will be proved in Lemma \ref{lem:area}.
The second integrality condition follows by taking $M=1$.

Conversely, we need to show that $\check{\bb{L}}_{K,b}$ fits into the exact sequence assuming the first and second integrality conditions. First of all, these conditions imply that
$$\det(\check{\bb{L}}_{K,b})\cong\det(\check{\bb{L}}_{1,b_1})\otimes\det(\check{\bb{L}}_{2,b_2}).$$
Since $r_1d_2>r_2d_1$, we have
$$\mu(\check{\bb{L}}_{K,b})=-\frac{d_1+d_2}{r_1+r_2}<-d_1\left(\frac{1+r_2/r_1}{r_1+r_2}\right)=-\frac{d_1}{r_1}=\mu(\check{\bb{L}}_{1,b_1}).$$
Also, the domains of $\bb{L}_1,\bb{L}_2$ and $\mathbb{L}_K$ are connected and $\text{gcd}(r_1,d_1)=\text{gcd}(r_2,d_2)=\text{gcd}(r_1+r_2,d_1+d_2)=1$, so $\check{\bb{L}}_1,\check{\bb{L}}_2,\check{\bb{L}}_K$ as well as $\check{\bb{L}}_{1,b_1},\check{\bb{L}}_{2,b_2},\check{\bb{L}}_{K,b}$ are all stable bundles. Hence we can apply Lemma \ref{lem:surjective} to find a surjective map $f:\check{\bb{L}}_{K,b}\to\check{\bb{L}}_{1,b_1}$.
Letting $\mathbb{K}:=\ker(f)$, we obtain the exact sequence
$$0\to\bb{K}\to\check{\bb{L}}_{K,b}\to\check{\bb{L}}_{1,b_1}\to 0.$$

Since every vector bundle over an elliptic curve is the flat limit of a family of semi-stable bundles with the same determinant (see \cite[Remark 1.1]{exact_sequence_stable_bundle}), we can choose families $\bb{K}(t)$ and $\check{\bb{L}}_{1,b_1}(t)$ such that $\bb{K}(0)=\bb{K}$ and $\check{\bb{L}}_{1,b_1}(0)=\check{\bb{L}}_{1,b_1}$.
By the classification result of Atiyah \cite{Atiyah_vector_bundle_over_an_elliptic_curve}, any indecomposable vector bundle on an elliptic curve with $\text{gcd}(\text{rk},\deg) = 1$ is determined by its determinant line bundle. Hence we have $\check{\bb{L}}_{1,b_1}(t)\cong\check{\bb{L}}_{1,b_1}$ for all $t$ near $0$ by the openness of semistability.

Since $\check{\bb{L}}_{K,b}$ is stable, we have $\text{Hom}(\check{\bb{L}}_{1,b_1},\bb{K})=0$ (see \cite[Lemma 1.1]{On_a_conjecture_of_Lange}).
Then we can apply Lemma \ref{lem:flat_family_limit} to obtain an exact sequence
\begin{equation}\label{equ:exact_sequence}
0\to\bb{K}(t)\to\check{\bb{L}}_{K,b}(t)\to\check{\bb{L}}_{1,b_1}(t)\to 0
\end{equation}
with $\check{\bb{L}}_{K,b}(0)=\check{\bb{L}}_{K,b}$.
Also, $\check{\bb{L}}_{K,b}(t)$ is semi-stable for all $t$ near $0$ again by openness of semistability. But
$$\det(\check{\bb{L}}_{K,b}(t))\cong\det(\bb{K}(t))\otimes\det(\check{\bb{L}}_{1,b_1}(t))=\det(\bb{K})\otimes\det(\check{\bb{L}}_{1,b_1})\cong\det(\check{\bb{L}}_{K,b}),$$
so we must have $\check{\bb{L}}_{K,b}(t)\cong\check{\bb{L}}_{K,b}$ for all $t$ near $0$.

Finally, note that
$$\det(\bb{K}(t))=\det(\bb{K})\cong\det(\check{\bb{L}}_{K,b})\otimes\det(\check{\bb{L}}_{1,b_1})^{-1}\cong\det(\check{\bb{L}}_{2,b_2})$$
and $\bb{K}(t)$ is semi-stable, so we have $\bb{K}(t)\cong\check{\bb{L}}_{2,b_2}$. Therefore, for small $t\neq 0$, the exact sequence (\ref{equ:exact_sequence}) reads
$$0\to\check{\bb{L}}_{2,b_2}\to\check{\bb{L}}_{K,b}\to\check{\bb{L}}_{1,b_1}\to 0.$$
This completes the proof of the theorem.
\end{proof}

Next we prove the integral formula that we need in the proof of Theorem \ref{thm:surgery_extension}.

\begin{lemma}\label{lem:area}
Let $\mathbb{L}_K$ be as in Theorem \ref{thm:surgery_extension} Then
$$\int_0^{r_1+r_2}\varphi(x)dx = \frac{r_1d_1}{2} + \frac{r_2d_2}{2} + c_1 + c_2 + N, \text{ for some $N\in\mathbb{Z}$},$$
where $\varphi:\mathbb{R}\to\mathbb{R}$ is the defining equation of $\mathbb{L}_K$.
\end{lemma}
\begin{proof}
First of all, by adding a sufficiently large integer, we may assume that $\varphi\geq 0$ on the interval $[0,r_1+r_2]$.

Observe that
$$\int_0^{r_1+r_2}\varphi(x)dx$$
is nothing but the area bounded by $\varphi$ and the $x$-axis, from $0$ to $r_1+r_2$. Since Lagrangian surgery is symmetric, the integral is the same as the area bounded by the piecewise linear function, obtained by replacing the non-linear portions by piecewise linear functions, with the $x$-axis, from $0$ to $r_1+r_2$.

We cut the area into several pieces as in Figure \ref{fig:area}.

\begin{figure}[H]
\centering
    \includegraphics[width=100mm]{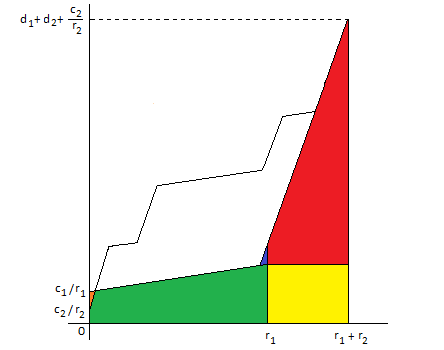}
    \caption{}
    \label{fig:area}
\end{figure}

Denote the area of the red, blue, yellow, green, orange and white region by $R,B,Y,G,O$ and $W$ respectively. Then we have
$$\int_0^{r_1+r_2}\varphi(x)dx=R+B+Y+G+W.$$
Note that we have the following
\begin{align*}
R & = \int_0^{r_2}\left(\frac{d_2}{r_2}x + \frac{c_2}{r_2}-\frac{c_1}{r_1}\right)dx = \frac{r_2d_2}{2}+c_2-\frac{c_1r_2}{r_1},\\
Y & = (d_1+\frac{c_1}{r_1})r_2 = d_1r_2 + \frac{c_1r_2}{r_1},\\
O + G & = \int_0^{r_1}\left(\frac{d_1}{r_1}x + \frac{c_1}{r_1}\right)dx = \frac{r_1d_1}{2} + c_1,\\
B & = O.
\end{align*}
Hence
$$\int_0^{r_1+r_2}\varphi(x)dx = \frac{r_1d_1}{2} + c_1 + \frac{r_2d_2}{2} + c_2 + d_1r_2 + W.$$
To see that $W$ is an integer, take a look at Figure \ref{fig:blue_red}.

\begin{figure}[H]
\centering
    \includegraphics[width=100mm]{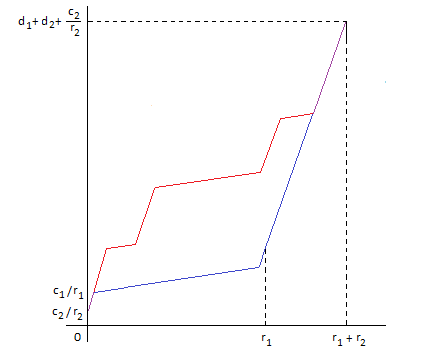}
    \caption{}
    \label{fig:blue_red}
\end{figure}

The red and blue lines can be viewed as two continuous maps $f_j:[0,1]\to X$, $j=1,2$ with the same image $\bb{L}_1\cup \bb{L}_2$. They define the same homology class in $H_1(X;\mathbb{Z})$, namely, $(r_1+r_2,d_1+d_2)\in\mathbb{Z}^2\cong H_1(X;\mathbb{Z})$. The white region serves as a 2-chain $\Delta$ such that
$$\partial\Delta = f_1 - f_2.$$
But $f_1,f_2$ also define the same image, so $\Delta$ is indeed a 2-cycle, i.e., $[\Delta]\in H_2(X;\mathbb{Z})$. Pulling back the symplectic form to $\mathbb{R}^2$, we then have $\omega=d(ydx)$. Let $\widetilde{f}_1,\widetilde{f}_2:[0,1]\to\mathbb{R}^2$ denote the lifts of $f_1,f_2$ starting at $c_1/r_2,c_2/r_2$ respectively. Then
$$W = \int_{\widetilde{f}_1 - \widetilde{f}_2} ydx = \int_{\Delta}\omega \in \mathbb{Z},$$
as $[\omega]$ is an integral class. This completes the proof of the lemma.
\end{proof}

Let us remark that the surgery-extension correspondence is {\em not} true for self-extension. For example, let $\bb{L}_0$ be the zero section and $\bb{L}'_0$ be the Lagrangian
$$\{(e^{2\pi ix},e^{2\pi i\sin(2\pi ix)}):x\in\bb{R}\},$$
which is Hamiltonian equivalent to $\bb{L}_0$. Hence both $\bb{L}_0$ and $\bb{L}_0'$ have $\cu{O}_{\check{X}}$, the trivial line bundle, as the SYZ mirror bundle. We perform a Lagrangian surgery at the index 1 intersection point $p$ to obtain an immersed Lagrangian multi-section $\bb{L}_p$ of rank 2. See Figure \ref{fig:self_extension}.

\begin{figure}[H]
\centering
    \includegraphics[width=100mm]{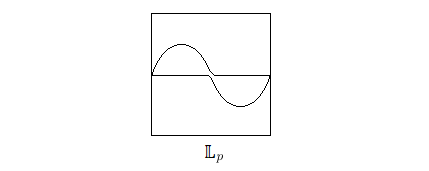}
    \caption{}
    \label{fig:self_extension}
\end{figure}

For any choice of local system $d+2\pi i\frac{b}{2}dx$ on the domain of $\bb{L}_p$, the transition function of $\check{\bb{L}}_{p,b}$ is given by
$$\begin{pmatrix}
0 & e^{2\pi ib}
\\ 1 & 0
\end{pmatrix}$$
Since it is a constant matrix, it is gauge equivalent to its diagonalization:
$$\begin{pmatrix}
e^{i\pi b} & 0
\\0 & -e^{i\pi b}
\end{pmatrix}.$$

If $\check{\bb{L}}_{p,b}$ is a self-extension of $\cu{O}_{\check{X}}$, then we have $e^{2i\pi b}=-1$, which is equivalent to the condition that $b\in\frac{1}{2}+\bb{Z}$. In this case, it is easy to see that $\check{\bb{L}}_{p,b}$ is isomorphic to a non-trivial decomposable holomorphic vector bundle of rank 2. However, it is known that the only self-extensions of $\cu{O}_{\check{X}}$ are $\cu{O}_{\check{X}}^{\oplus 2}$ and the Atiyah bundle $\check{A}_2$, which is indecomposable. Therefore $\check{\bb{L}}_{p,b}$ cannot be a self-extension of $\cu{O}_{\check{X}}$ for any choice of $b$.

Accordingly we expect that $\bb{L}_{p,b}$ is not a mapping cone of $p:\bb{L}_0\to \bb{L}_0'$. Note that this does not violate the result of Abouzaid \cite{Abouzaid_Fukaya_categories_of_higher_genus} because he required two curves to be intersecting minimally within their isotopy class. In our example, $\bb{L}_0'$ does not intersect $\bb{L}_0$ minimally within its isotopy class (the minimal intersection is in fact empty).

\begin{remark}
As we have mentioned in the introduction, the integrality condition
$$b_1+b_2-b-\frac{1}{2}\in\bb{Z}$$
suggests that given two Lagrangian A-branes $(\bb{L}_1,\cu{L}_1),(\bb{L}_2,\cu{L}_2)$ in a symplectic manifold $(M,\omega)$, if $\cu{L}_1$, $\cu{L}_2$ have holonomies $e^{2\pi ib_1}$, $e^{2\pi ib_2}$ respectively, then the holonomy $e^{2\pi ib}$ of $\cu{L}$ should be chosen to satisfy
$$e^{2\pi ib}e^{-2\pi ib_1}e^{-2\pi ib_2}=-1.$$
We believe that this relation can be understood in terms of Floer theory.
\end{remark}

\section{Invariance of immersed Floer cohomology}\label{sec:inv_thm}

As we have seen in the introduction, the surgery-extension correspondence theorem (Theorem \ref{thm:surgery_extension}) gives us a pair of non-Hamiltonian equivalent Lagrangian immersions that share the same SYZ mirror bundle. By homological mirror symmetry, we expect that the two Lagrangian immersions should be equivalent in the immersed Fukaya category. Indeed it was pointed out by Akaho and Joyce in their work \cite{AJ} on immersed Floer theory that the immersed Floer cohomology should have an invariance property under some equivalence which is weaker than global Hamiltonian equivalence. In this section, we will define a new notion called lifted Hamiltonian equivalence and prove that the immersed Floer cohomology is invariant under this new equivalence. So let us digress from mirror symmetry for a moment and turn our attention to symplectic geometry.

Throughout this section, the notation $(M,\omega)$ will stand for a $2n$-dimensional compact symplectic manifold equipped with a symplectic form $\omega$. The notion of Lagrangian immersions is defined as in Definition \ref{def:Lag_imm}. We always assume that the domain of a Lagrangian immersion is compact. We first recall the definition of immersed Floer cohomology for a pair of Lagrangian immersions introduced in \cite{AJ} by Akaho and Joyce.

\subsection{Maslov index and the immersed Floer cohomology}

To have good Floer theory, Akaho and Joyce made the following assumption:
\begin{assumption}\label{assumption:A}
The intersection points of $\xi_1(L_1)$ and $\xi_2(L_2)$ are finite and do not coincide with their self-intersection points.
\end{assumption}
The Floer complex of two transversally-intersecting immersed Lagrangians $\mathbb{L}_1=(L_1,\xi_1)$ and $\mathbb{L}_2=(L_2,\xi_2)$ is defined by
$$CF(\mathbb{L}_1,\mathbb{L}_2):=\bigoplus_{p\in\xi_1(L_1)\cap\xi_2(L_2)}\Lambda_{nov}\cdot p,$$
where $\Lambda_{nov}$ is the Novikov field given by
$$\Lambda_{nov}:=\left\{ \sum_{i=1}^{\infty}a_iT^{\lambda_i}: a_i\in k,\lim_{i\to\infty}\lambda_i = \infty \right\}$$
and $k$ is a field ($\bb{R}$, $\bb{C}$ or $\bb{Z}_2$).

Let $\{J_t\}_{t\in[0,1]}$ be a family of almost complex structure on $M$ which are compatible with $\omega$. Let $g_t(\cdot,\cdot):=\omega(J_t\cdot,\cdot)$ be the Riemannian metric associated to the pair $(\omega,J_t)$. The Floer differential $m_1$ is defined by counting $J_t$-holomorphic strips $u:\bb{R}\times[0,1]\to M$:
$$\pd{u}{s}+J_t(u)\pd{u}{t}=0$$
with finite energy
$$\int_{\bb{R}\times[0,1]}|du(s,t)|_{g_t}dsdt<+\infty$$
and boundary data
$$u(\bb{R}\times\{0\})\subset\xi_1(L_1),\quad u(\bb{R}\times\{1\})\subset\xi_2(L_2),$$
$$\lim_{s\to -\infty}u(s,t)=q,\quad \lim_{s\to +\infty}u(s,t)=p.$$
Since $\bb{L}_1,\bb{L}_2$ are immersed, they require that there are continuous liftings $u_1^-:\bb{R}\times\{0\}\to L_1$, $u_2^+:\bb{R}\times\{1\}\to L_2$ such that $\xi_1\circ u_1^-=u|_{\bb{R}\times\{0\}}$ and $\xi_2\circ u_2^+=u|_{\bb{R}\times\{1\}}$ (see Figure \ref{fig:count}).

\begin{figure}[H]
\centering
    \includegraphics[width=100mm]{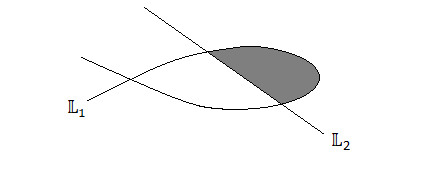}
    \caption{A $J_t$-holomorphic strip bounded by $\bb{L}_1$ and $\bb{L}_2$.}
    \label{fig:count}
\end{figure}

Let $\pi_2(M;\xi_1(L_1),\xi_2(L_2);p,q)$ be the space of all homotopy classes of strips $u:\bb{R}\times[0,1]\to M$ that satisfy the above boundary conditions and lifting properties. Fix a homotopy class $\beta\in\pi_2(M;\xi_1(L_1),\xi_2(L_2);p,q)$. Let $\til{\cu{M}}(p,q;\beta)$ be the moduli space of all $J_t$-holomorphic disks that represent the class $\beta$ and satisfy the above boundary data. Quotienting by the action of translation in the $s$-direction, we obtain the moduli space
$$\cu{M}(p,q;\beta):=\til{\cu{M}}(p,q;\beta)/\bb{R}.$$

In \cite{FOOO1, FOOO2}, the authors proved that $\cu{M}(p,q;\beta)$ is a Kuranishi space and can be compactified. Moreover, if the Lagrangian immersions are relatively spin, then the compactified moduli space can be oriented.
To describe the dimension of the moduli space, one needs to introduce the Maslov index; an excellent introduction of Maslov index can be found in Auroux's article \cite{Auroux_Beginner_intro_to_Fuk}.
Choose a symplectic trivialization $\Phi:u^*TM\cong(\bb{R}\times[0,1])\times T_pM$. Consider the Lagrangian paths
\begin{align*}
& \gamma_1^-:s\mapsto \Phi(d\xi_1(T_{u_1^-(-s,0)}L_1))\subset T_pM,\\
& \gamma_2^+:s\mapsto \Phi(d\xi_2(T_{u_2^+(s,1)}L_2))\subset T_pM
\end{align*}
in the Lagrangian Grassmanian $LGr(T_pM,\omega_p)$. Since $p$ is not a self-intersection point for either $\bb{L}_1,\bb{L}_2$, it has unique preimage points $l_1\in L_1$ and $l_2\in L_2$. We identify $(T_pM,\omega_p)$ with $(\bb{C}^n,\omega_{std})$. There exists $A\in Sp(2n,\bb{R})$ such that
$$A(d\xi_1(T_{l_1}L_1))=\bb{R}^n,\quad A(d\xi_2(T_{l_2}L_2))=i\bb{R}^n.$$
The canonical short path $\lambda_p:[0,1]\to LGr(T_pM,\omega_p)$ is defined to be
$$\lambda_p(t):=A^{-1}(e^{-\frac{\pi i}{2}t}(i\bb{R}^n)).$$
Then by concatenating the paths, $\gamma:=\lambda_p^{-1}*\gamma_1^-*\Phi(\lambda_q)*\gamma_2^+$ defines a loop in $LGr(T_pM,\omega_p)$, based at $d\xi_1(T_{l_1}L_1)$. Recall that $\pi_1(LGr(T_pM,\omega_p))\cong\bb{Z}$.
\begin{definition}
The {\em Maslov index} $\mu(u)$ of the strip $u:\bb{R}\times[0,1]\to M$ is defined as the degree of the loop $\gamma\subset LGr(T_pM,\omega_p)$.
\end{definition}
It is a well-known fact that the Maslov index only depends on the homotopy class of the strip $u:\bb{R}\times[0,1]\to M$. The (virtual) dimension of $\cu{M}(p,q;\beta)$ is given by $\mu(\beta)-1$.

The Floer differential $m_1:CF(\mathbb{L}_1,\mathbb{L}_2)\to CF(\mathbb{L}_1,\mathbb{L}_2)$ is defined by
$$m_1(p):=\sum_{q\in\xi_1(L_1)\cap\xi_2(L_2)}\sum_{\beta:\mu(\beta)=1}\sum_{u\in\cu{M}(p,q;\beta)}(-1)^{sign(u)}T^{\omega(u)}\cdot q,$$
where the sign $(-1)^{\text{sign}(u)}$ is determined by the orientation of the moduli space $\cu{M}(p,q;\beta)$ and
$$\omega(u):=\int_{\bb{R}\times[0,1]}u^*\omega$$
is the symplectic area of $u$. By Gromov compactness, the sum converges in $\Lambda_{nov}$.

\begin{definition}
If $m_1^2=0$, then the {\em immersed Floer cohomology} is defined as
$$HF(\mathbb{L}_1,\mathbb{L}_2):=H(CF(\bb{L}_1,\bb{L}_2),m_1).$$
\end{definition}

We assume all Lagrangian immersions we consider here are {\em unobstructed}, meaning that the Floer differential $m_1$ satisfies $(m_1)^2=0$.

\begin{remark}
Usually, the notion of unobstructed Lagrangian immersion involves a bounding cochain $b$ on the domain of the immersion. In this paper, we will consider those Lagrangian immersions with $b=0$.
\end{remark}

It is well known that the two limits
$$\lim_{s\to -\infty}u(s,t)=q,\quad \lim_{s\to +\infty}u(s,t)=p$$
converge uniformly in $t\in [0,1]$. Indeed, one has
$$\text{dist}(u(s,t),p)<Ce^{-\mu|s|}\ \text{ for all }t\in[0,1],$$
where $C,\mu>0$ are constants depending only on the energy $E(u)$ of $u$. By identifying $\bb{R}\times[0,1]$ with the closed unit disk $\Delta$ with punctures at $\pm 1$, the limits $\lim_{z\to -1}u(z),\lim_{z\to +1}u(z)$ exist and are equal to $q,p$ respectively. Therefore, it makes sense to write $u(-1)=q$ and $u(1)=p$.

From now on, we replace the strip model by the disk model with finite energy and boundary data
$$u(\partial^-\Delta)\subset\xi_1(L_1),\quad u(\partial^+\Delta)\subset\xi_2(L_2),$$
$$u(-1)=q,\quad u(1)=p.$$
Here, we put $\partial^-\Delta=S^1\cap\{z\in\bb{C}:\text{Im}(z)\leq 0\}$ and $\partial^+\Delta=S^1\cap\{z\in\bb{C}:\text{Im}(z)\geq 0\}$.

\subsection{Three types of equivalences}

First, we recall the notion of global Hamiltonian equivalence introduced in \cite{AJ}.

\begin{definition}\label{def:global_Ham}
Let $(L_1,\xi_2)$, $(L_2,\xi_2)$ be two Lagrangian immersions in a symplectic manifold $(M,\omega)$. They are said to be {\em globally Hamiltonian equivalent} if there exists a diffeomorphism $\phi:L_1\to L_2$ and a 1-parameter family of Hamiltonian diffeomorphism $\psi_t:M\to M$ such that $\psi_0=id_M$ and $\psi_1\circ\xi_1=\xi_2\circ\phi$.
\end{definition}


Akaho and Joyce proved that $HF(\mathbb{L}_1,\mathbb{L}_2)$ is indeed a global Hamiltonian invariant, that is, if $\bb{L}_2$ is globally Hamiltonian isotopic to $\bb{L}_2'$, then there is a quasi-isomorphism
$$(CF(\bb{L}_1,\bb{L}_2),m_1)\simeq(CF(\bb{L}_1,\bb{L}_2'),m_1').$$
In the immersed situation, there is another equivalence called local Hamiltonian equivalence. Let us recall its definition.

\begin{definition}\label{def:loc_Ham}
Let $(L_1,\xi_1)$, $(L_2,\xi_2)$ be two Lagrangian immersions in a symplectic manifold $(M,\omega)$. They are said to be {\em locally Hamiltonian equivalent} if there exists a diffeomorphism $\phi:L_1\to L_2$ and a smooth 1-parameter family $\Xi:[0,1]\times L_1\to M$ such that $\Xi(0,-)=\xi_1$, $\Xi(1,-)=\xi_2\circ\phi$ and the 1-form
$$\Xi^*(\omega)(\frac{d}{dt},-)$$
on $\{t\}\times L_1$, is exact for all $t\in [0,1]$.
\end{definition}

Note that local Hamiltonian equivalence can be implied by global Hamiltonian equivalence by pulling back the Hamiltonian function on $M$ to the domain $L$ via the immersion $\xi$ but not the converse in general, as shown by the example below.

\begin{example}
Consider the Lagrangian immersions $\bb{L},\bb{L}'$ in the standard symplectic 2-torus $T^2$ as shown in Figure \ref{fig:loc_ham}.

\begin{figure}[H]
\centering
    \includegraphics[width=100mm]{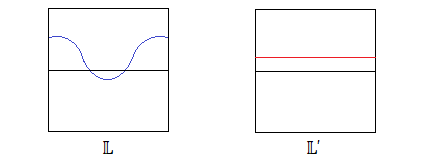}
    \caption{}
    \label{fig:loc_ham}
\end{figure}

Clearly, $\bb{L}$ and $\bb{L}'$ are not globally Hamiltonian equivalent as they share a different number of self-intersection points. Nevertheless, the blue curve can be Hamiltonian deformed into a horizontal Lagrangian section, namely, the red line. Hence $\bb{L}$ and $\bb{L}'$ are in fact locally Hamiltonian equivalent.
\end{example}

A natural question is to ask whether $HF(\bb{L}_1,\bb{L}_2)$ is invariant under local Hamiltonian equivalence. It was pointed out by Akaho and Joyce that this is not true for general local Hamiltonian isotopies (see \cite[Section 13]{AJ}). The reason behind this is the Lagrangian $h$-principle \cite{Gromov_Lagrangian_h, Lee_Lagrangian_h}, which states that two Lagrangian immersions $(L,\xi_1)$, $(L,\xi_2)$ are locally Hamiltonian equivalent if and only if there exists a smooth homotopy $\xi_t:L\to M$ from $(\xi_1,d\xi_1)$ to $(\xi_2,d\xi_2)$ and a bundle map $\til{\xi}_t:TL\to \xi^*_tTM$ covering $\xi_t$ which embeds $TL$ as a Lagrangian subbundle in $\xi_t^*TM$. But $HF(\bb{L}_1,\bb{L}_2)$ consists of quantum data coming from holomorphic disks, which is invisible to classical algebraic topology, so one would not expect these quantum data to be preserved under general local Hamiltonian isotopies.

Our goal is to find a new equivalence which is weaker than global Hamiltonian equivalence, but stronger than local Hamiltonian equivalence, such that $HF(\bb{L}_1,\bb{L}_2)$ is invariant under this equivalence.
Let us start with the following
\begin{definition}\label{def:M_lifted_Ham}
Let $\pi:\til{M}\to M$ be a finite unramified covering of a symplectic manifold $(M,\omega)$. For two Lagrangian immersions $\mathbb{L}_1 = (L_1, \xi_1), \mathbb{L}_2 = (L_2, \xi_2)$ of $M$, we say $\mathbb{L}_1$ is {\em $(\til{M},\pi)$-lifted Hamiltonian isotopic} to $\mathbb{L}_2$ if there exists a diffeomorphism $\phi:L_1\to L_2$ and Lagrangian immersions $\til{\xi}_1:L_1\to\til{M}$, $\til{\xi}_2:L_2\to\til{M}$ such that $\xi_1=\pi\circ\til{\xi}_1$, $\xi_2=\pi\circ\til{\xi}_2$ and $(L_1,\til{\xi}_1)$ is globally Hamiltonian isotopic to $(L_1,\til{\xi}_2\circ\phi)$ in $(\til{M},\pi^*\omega)$.
\end{definition}

We remark that $L_1, L_2$ and $\til{M}$ can all be disconnected. When $\mathbb{L}_1$ is $(\til{M},\pi)$-lifted Hamiltonian isotopic to $\mathbb{L}_2$, we may assume that the immersions share the same domain, i.e., $L_1=L_2$. In this case we may take $\phi$ to be the identity map.

Note that Definition \ref{def:M_lifted_Ham} does not define an equivalence relation because the relation that $\bb{L}_1$ is $(\til{M},\pi)$-lifted Hamiltonian isotopic to $\bb{L}_2$ for some finite unramified covering $\pi:\til{M}\to M$ is {\em not} transitive. So we need to make the following
\begin{definition}\label{def:lifted_Ham}
Let $(M,\omega)$ be a symplectic manifold. For two Lagrangian immersions $\mathbb{L}_1 = (L_1, \xi_1), \mathbb{L}_2 = (L_2, \xi_2)$ of $M$, we say $\mathbb{L}_1$ is {\em lifted Hamiltonian isotopic} to $\mathbb{L}_2$ if there exists an integer $l>0$ and Lagrangian immersions $\bb{L}^{(1)}:=\bb{L}_1,\bb{L}^{(2)},\dots,\bb{L}^{(l-1)},\bb{L}^{(l)}:=\bb{L}_2$ of $M$, such that $\bb{L}^{(j)}$ is $(\til{M}_j,\pi_j)$-lifted Hamiltonian isotopic to $\bb{L}^{(j+1)}$, for some finite unramified covering $\pi_j:\til{M}_j\to M$, $j=1,\dots,l-1$.
\end{definition}

Clearly, lifted Hamiltonian isotopy defines an equivalence relation on the set of Lagrangian immersions. Hence it makes sense to say that $\bb{L}_1$ is {\em lifted Hamiltonian equivalent} to $\bb{L}_2$. Note that two Lagrangian immersions $\bb{L}_1,\bb{L}_2$ are globally Hamiltonian equivalent if and only if they are $(M,id_M)$-lifted Hamiltonian isotopic to each other.

\begin{proposition}
If $\mathbb{L}_1$ and $\mathbb{L}_2$ are Lagrangian immersions which are lifted Hamiltonian equivalent, then $\mathbb{L}_1$ and $\mathbb{L}_2$ are locally Hamiltonian equivalent.
\end{proposition}
\begin{proof}
It suffices to prove that if $\mathbb{L}_1$ and $\mathbb{L}_2$ are $(\til{M},\pi)$-lifted Hamiltonian isotopic for some finite unramified covering $\pi:\til{M}\to M$, then they are locally Hamiltonian equivalent. Let $\widetilde{\xi}_1,\widetilde{\xi}_2:L\to\til{M}$ be lifts of $\xi_1,\xi_2$ respectively.
By assumption, there exists a family of Hamiltonian diffeomorphisms $\widetilde{\Xi}_t:\til{M}\to\til{M}$ such that $\widetilde{\Xi}_0=id$ and $\widetilde{\Xi}_1\circ\widetilde{\xi}_1=\widetilde{\xi}_2$. Since $\pi:\til{M}\to M$ is an unramified covering, $\Xi_t:=\pi\circ\widetilde{\Xi}_t\circ\widetilde{\xi}_1:L\to M$ defines a family of Lagrangian immersions such that
\begin{align*}
\Xi_0 & = \pi\circ\widetilde{\Xi}_0\circ\widetilde{\xi}_1 = \xi_1,\\
\Xi_1 & = \pi\circ\widetilde{\Xi}_1\circ\widetilde{\xi}_1 = \xi_2.
\end{align*}
We shall prove that
$$\Xi^*(\omega)(\frac{d}{dt},-)$$
is exact on $\{t\}\times L$ for all $t\in[0,1]$. Let $h:[0,1]\times\til{M}\to\mathbb{R}$ be a Hamiltonian that generate $\widetilde{\Xi}_t$. We claim that
$$d_L(h_t\circ\widetilde{\Xi}_t\circ\til{\xi}_1)=\Xi^*(\omega)(\frac{d}{dt},-).$$
For any $v\in\Gamma(L,TL)$,
\begin{align*}
\Xi^*(\omega)(\frac{d}{dt},v) & = \omega(\Xi_*\frac{d}{dt},\Xi_*v)\\
& = (\pi^*\omega)((\til{\Xi}\circ\til{\xi}_1)_*\frac{d}{dt},(\til{\Xi}_t\circ\til{\xi}_1)_*v)\\
& = (\pi^*\omega)(X_{h_t}(\til{\Xi}_t\circ\til{\xi}_1),(\til{\Xi}_t\circ\til{\xi}_1)_*v)\\
& = d_{\til{M}}h_t((\til{\Xi}_t\circ\til{\xi}_1)_*v)\\
& = d_L(h_t\circ\widetilde{\Xi}_t\circ\til{\xi}_1)(v),
\end{align*}
so we are done.
\end{proof}

As a summary, we have

\begin{corollary}\label{cor:property_lifted_Ham}
Let $\mathbb{L}_1$, $\mathbb{L}_2$ Lagrangian immersions. Consider the following statements:
\begin{itemize}
\item [a)] $\mathbb{L}_1$ and $\mathbb{L}_2$ are globally Hamiltonian equivalent.
\item [b)] $\mathbb{L}_1$ and $\mathbb{L}_2$ are lifted Hamiltonian equivalent.
\item [c)] $\mathbb{L}_1$ and $\mathbb{L}_2$ are locally Hamiltonian equivalent.
\end{itemize}
Then we have the implications $a)\Rightarrow b)\Rightarrow c)$.
\end{corollary}

\begin{remark}
When $\bb{L}_1,\bb{L}_2$ are embedded and locally Hamiltonian equivalent, they are Hamiltonian isotopic to each other if we can choose the isotopy $\Xi_t$ to be an embedding for all $t\in [0,1]$. Hence the statements $a), b), c)$ in Corollary \ref{cor:property_lifted_Ham} are all equivalent in this case.
\end{remark}

\subsection{The invariance theorem}

We study the invariance property of the immersed Floer cohomology under lifted Hamiltonian deformations.

Let $\mathbb{L}_1,\mathbb{L}_2$ be a pair of compact, unobstructed Lagrangian immersions of $(M,\omega)$. Let $\til{\xi}_1:L_1\to\til{M}_1$ and $\til{\xi}_2:L_2\to\til{M}_2$ be liftings of $\bb{L}_1,\bb{L}_2$ to some finite unramified coverings $\pi_1:\til{M}_1\to M$ and $\pi_2:\til{M}_2\to M$ of $M$ respectively. Note that $(\til{M}_j,\pi_j)$ may be equal to $(M,id_M)$, that is, the trivial covering of $M$.

Consider the following commutative diagram:
\begin{equation*}
\xymatrix{
L_1\times_M\til{M}_2 \ar[d]_{\pi_{L_1}} \ar@{->}[r]^{\til{\xi}_1\times id} & {\til{M}_1\times_M\til{M}_2} \ar[d]_{\pi_M} \ar@{<-}[r]^{id\times\til{\xi}_2}& \til{M}_1\times_ML_2 \ar[d]_{\pi_{L_2}}\\
L_1 \ar@{->}[r]^-{\til{\xi}_1} & {\til{M}_1\longrightarrow M \longleftarrow\til{M}_2} \ar@{<-}[r]^-{\til{\xi}_2} & L_2
}
\end{equation*}
Note that all the vertical maps are finite unramified covering maps (the domains of them can be disconnected in general, but they are still smooth manifolds).

Since $(L_1,\xi_1)$ and $(L_2,\xi_2)$ are Lagrangian immersions, by equipping $\til{M}_1\times_M\til{M}_2$ with the pullback symplectic structure via $\pi_M$, it is easy to see that $(L_1\times_M\til{M}_2,\til{\xi}_1\times id)$ and $(\til{M}_1\times_ML_2,id\times\til{\xi}_2)$ are Lagrangian immersions of $\til{M}_1\times_M\til{M}_2$, whose images are given by $\til{\xi}_1(L_1)\times_M\til{M}_2$ and $\til{M}_1\times_M\til{\xi}_2(L_2)$ respectively.

To simplify the notation, we let $\til{M}=\til{M}_1\times_M\til{M}_2$, $\til{L}_1=\til{\xi}_1(L_1)\times_M\til{M}_2$ and $\til{L}_2=\til{M}_1\times_M\til{\xi}_2(L_2)$. Points in $\til{M}$ are denoted by $\til{p}$.

\begin{lemma}\label{lem:intersection_points_correspondence}
Under Assumption \ref{assumption:A}, we have
\begin{itemize}
\item [a)] $\pi_M:\til{L}_1\cap\til{L}_2\to\xi_1(L_1)\cap\xi_2(L_2)$ is a 1-1 correspondence.
\item [b)] The map $(\pi_M)_*:\pi_2(\til{M};\til{L}_1,\til{L}_2;\til{p},\til{q})\to\pi_2(M;\xi_1(L_1),\xi_2(L_2);p,q)$ is bijective.
\item [c)] $(\pi_M)_*$ preserves the Maslov index, i.e., $\mu(\til{\beta})=\mu((\pi_M)_*\til{\beta})$.
\item [d)] If  $(M,\omega,J)$ is a Calabi-Yau manifold, then $\pi_M$ preserves grading, i.e., for any $\til{p}\in\til{L}_1\cap\til{L}_2$, $\deg(\pi_M(\widetilde{p}))=\deg(\widetilde{p})$.
\end{itemize}
\end{lemma}
\begin{proof}

\begin{itemize}
\item [a)]
Suppose $p\in\xi_1(L_1)\cap \xi_2(L_2)$. Then there exists $l_1\in L_1$ and $l_2\in L_2$ such that $\xi_1(l_1)=p=\xi_2(l_2)$. Since $(\pi_1\circ\til{\xi}_1)(l_1)=\xi_1(l_1)=p=(\pi_2\circ\til{\xi}_2)(l_2)$, we have $$(\til{\xi}_1(l_1);p;\til{\xi}_2(l_2))\in\til{L}_1\cap\til{L}_2$$
and $\pi_M(\til{\xi}_1(l_1);p;\til{\xi}_2(l_2))=p$. This proves surjectivity. For injectivity, note that any intersection point of $\til{L}_1$ and $\til{L}_2$ is of the form $(\til{\xi}_1(l_1);p;\til{\xi}_2(l_2))$ for some $l_1\in L_1,l_2\in L_2$ and $\xi_1(l_1)=p=\xi(l_2)$. If
$$\pi_M(\til{\xi}_1(l_1);p;\til{\xi}_2(l_2))=\pi_M(\til{\xi}_1(l_1');p';\til{\xi}_2(l_2')),$$
then $p=p'$ and so
\begin{align*}
\xi_1(l_1) & = p = p' = \xi_1(l_1'),\\
\xi_2(l_2) & = p = p' = \xi_2(l_2').
\end{align*}
Since $p$ is not a self-intersection point of $\xi_1(L_1)$ nor $\xi_2(L_2)$, we have $l_1=l_1'$ and $l_2=l_2'$. Hence $(\til{\xi}_1(l_1);p;\til{\xi}_2(l_2))=(\til{\xi}_1(l_1');p';\til{\xi}_2(l_2'))$.

\item [b)]
We first prove that $(\pi_M)_*$ is well-defined, i.e., the image of each disk under $\pi_M$ satisfies the required boundary data.

Let $\til{u}:\Delta\to\til{M}$ represent $\til{\beta}$ with boundary data
$$\til{u}(\partial^-\Delta)\subset\til{L}_1,\quad \til{u}(\partial^+\Delta)\subset\til{L}_2,$$
$$\til{u}(-1)=\til{q},\quad \til{u}(1)=\til{p}.$$
Set $u:=\pi_M\circ\til{u}$. Then clearly, $u$ has boundary data
$$u(\partial^-\Delta)\subset\xi_1(L_1),\quad u(\partial^+\Delta)\subset\xi_2(L_2),$$
$$u(-1)=q,\quad u(1)=p.$$
To obtain the liftings on the boundary, we recall we already have the liftings $\til{u}_1^-:\partial^-\Delta\to L_1\times_M\til{M}_2$ and $\til{u}_2^+:\partial^+\Delta\to\til{M}_1\times_BL_2$ of $\til{u}|_{\partial^-\Delta}$ and $\til{u}|_{\partial^+\Delta}$ respectively. By definition, they satisfy
\begin{align*}
(\til{\xi}_1\times id)\circ\til{u}_1^- & = \til{u}|_{\partial^-\Delta},
\\
(id \times\til{\xi}_2)\circ\til{u}_2^+ & = \til{u}|_{\partial^+\Delta}.
\end{align*}
We define
\begin{align*}
u_1^- & := \pi_{L_1}\circ\til{u}_1^-:\partial^-\Delta\to L_1,\\
u_2^+ & := \pi_{L_2}\circ\til{u}_2^+:\partial^+\Delta\to L_2.
\end{align*}
Then $\xi_1\circ u_1^-=\xi_1\circ\pi_{L_1}\circ\til{u}_1^-=\pi_M\circ(\til{\xi}_1\times id)\circ\til{u}_1^-=\pi_M\circ\til{u}|_{\partial^-\Delta}=u|_{\partial^-\Delta}$. Similarly, $\xi_2\circ u_2^+=u|_{\partial^+\Delta}$. Hence $(\pi_M)_*$ is well-defined. Injectivity follows from the homotopy lifting property.

For surjectivity, let $u:\Delta\to M$ be a representative of $\beta$ with boundary data
$$u(\partial^-\Delta)\subset\xi_1(L_1),\quad u(\partial^+\Delta)\subset\xi_2(L_2),$$
$$u(-1)=q,\quad u(1)=p.$$
Let $\til{u}:\Delta\to\til{M}$ be the lift of $u$ with $\til{u}(1)=\til{p}$. We claim that $\til{u}(\partial^-\Delta)\subset\til{L}_1$. Recall that we have a lift $u_1^-:\partial^-\Delta\to L_1$ of $u|_{\partial^-\Delta}$. Since $\pi_{L_1}:L_1\times_M\til{M}_2\to L_1$ is an unramified covering of $L_1$ and $\xi_1^{-1}(p)$ consists of only one point, there is a lift $\til{u}_1^-:\partial^-\Delta\to L_1\times_M\til{M}_2$ of $u_1^-$ such that $((\til{\xi}_1\times id)\circ\til{u}_1^-)(1)=\til{p}$. Note that
$$\pi_{\til{M}_1}\circ(\til{\xi}_1\times id)\circ\til{u}_1^-=\til{\xi}_1\circ\pi_{L_1}\circ\til{u}_1^-=\til{\xi}_1\circ u_1^-.$$
Hence
\begin{align*}
\pi_M\circ(\til{\xi}_1\times id)\circ\til{u}_1^-=&\pi_1\circ\pi_{\til{M}_1}\circ(\til{\xi}_1\times id)\circ\til{u}_1^-
\\=&\xi_1\circ u_1^-
\\=&u|_{\partial^-\Delta}=\pi_M\circ\til{u}|_{\partial^-\Delta}.
\end{align*}
By uniqueness, we have $\til{u}|_{\partial^-\Delta}=(\til{\xi}_1\times id)\circ\til{u}_1^-$. In particular, we have $\til{u}(\partial^-\Delta)\subset\til{L}_1$. Similarly, we have $\til{u}(\partial^+\Delta)\subset\til{L}_2$. These two inclusions imply
$$\til{u}(-1)\in\til{u}(\partial^+\Delta\cap\partial^-\Delta)\subset\til{L}_1\cap\til{L}_2.$$
Because $u(-1)=q$, we must have $\til{u}(-1)=\til{q}$ by uniqueness.

\item [c)] Since, via the differential $d\pi_M:T\til{M}\to\pi_M^*TM$, $T_{\til{q}}\til{M}$ can be identified symplectically with $T_qM$, the Lagrangian Grassmannians $LGr(T_{\til{q}}\til{M},\til{\omega}_{\til{q}})$ and $LGr(T_qM,\omega_q)$ are naturally isomorphic via $d\pi_M$.Clearly, the Lagrangian paths
\begin{align*}
&s\mapsto d\til{\xi_1}(T_{\til{u}_1^-(s,0)}(L_1\times_M\til{M}_2)),
\\&s\mapsto d\xi_1(T_{u_1^-(s,0)}L_1)
\end{align*}
can also be identified via $d\pi_M$. Similarly, $d\pi_M$ also identifies
\begin{align*}
&s\mapsto d\til{\xi_2}(T_{\til{u}_2^+(s,1)}(\til{M}_1\times_ML_2)),
\\&s\mapsto d\xi_1(T_{u_2^+(s,1)}L_2).
\end{align*}
Since $\pi_M$ is an unramified covering map, it is a local symplectomorphism. It follows that the canonical short paths are also identified via $d\pi_M$. Hence the Maslov indices are preserved under $(\pi_M)_*$.

\item [d)]
Since $\pi_M$ is an unramified covering map, $(\til{M},\til{\omega},\til{J})$ is naturally a Calabi-Yau manifold with the pullback structures $\til{\omega}=\pi_M^*\omega$, $\til{J}=\pi_M^*J$ and so each intersection point between $\til{L}_1$ and $\til{L}_2$ can be graded. Let $\widetilde{p}\in\til{L}_1\cap\til{L}_2$. The grading of $\widetilde{p}$ only depends on the angles that $\til{L}_1$ and $\til{L}_2$ intersect (See \cite{AJ}, Section 12), which is a local property. Since $\pi$ is an unramified $(\widetilde{J},J)$-holomorphic covering map, we see that the angles of intersection at $\widetilde{p}$ is the same as the angles of intersection at $\pi_M(\widetilde{p})$.
\end{itemize}
\end{proof}

Next, we show that the immersed Floer cohomology $HF(\bb{L}_1,\bb{L}_2)$ can be computed by the immersed Floer cohomology of the liftings $(L_1\times_M\til{M}_2,\til{\xi}_1\times id),(\til{M}_1\times_ML_2,id\times\til{\xi}_2)$. Recall that we have chosen a family of $\omega$-compatible almost complex structures $\{J_t\}_{t\in[0,1]}$. Let $\{\til{J}_t\}_{t\in[0,1]}$ be the pullback almost complex structure of $\{J_t\}_{t\in[0,1]}$ via the unramified covering map $\pi_M:\til{M}\to M$. Then we have the following

\begin{lemma}
For any $p,q\in\xi_1(L_1)\cap\xi_2(L_2)$, $\beta\in\pi_2(M;L_1,L_2;p,q)$ and $u\in\cu{M}(p,q;\beta)$,
there exist unique $\til{p},\til{q}\in\til{L}_1\cap\til{L}_2$, $\til{\beta}\in\pi_2(\til{M};\til{L}_1,\til{L}_2;\til{p},\til{q})$ and $\til{u}\in\cu{M}(\til{p},\til{q};\til{\beta})$ such that $\pi_M(\til{p})=p$, $\pi_M(\til{q})=q$ and $\pi_M\circ\til{u}=u$. Moreover, $(\pi_M)_*$ induces an orientation preserving isomorphism of oriented Kuranishi spaces
$$\cu{M}(p,q;\beta) \cong \cu{M}(\til{p},\til{q};\til{\beta}).$$
\end{lemma}
\begin{proof}
The existence and uniqueness of $\til{p},\til{q}$ follow from Lemma \ref{lem:intersection_points_correspondence}. Also, since $\pi_M$ is an unramified $(\til{J}_t,J_t)$-holomorphic covering map, the proof of Part b) of Lemma \ref{lem:intersection_points_correspondence} has already yielded the correspondence between $\til{J}_t$-holomorphic disks in $\til{M}$ and $J_t$-holomorphic disks in $M$ with the given boundary data and lifting properties.

So far, we see that $(\pi_M)_*$ gives a bijection between the sets $\cu{M}(p,q;\beta)$ and $\cu{M}(\til{p},\til{q};\til{\beta})$. By \cite{FOOO2}, the Kuranishi structure of $\cu{M}(p,q;\beta)$ is governed by the linearized Cauchy-Riemann operator
$$D_u\dbar_{J_t}:W^{1,p}(u^*TM;u|_{\partial^-\Delta}^*d\xi_1(TL_1),u|_{\partial^+\Delta}^*d\xi_2(TL_2))\to L^p(u^*TM)$$
at every $u\in\cu{M}(p,q;\beta)$.
In order to prove that we have an isomorphism of Kuranishi spaces, we need to show that the linearized Cauchy-Riemann operator $D_{\til{u}}\dbar_{\til{J}_t}$ can be identified with $D_u\dbar_{J_t}$ every point $\til{u}$. Since $\pi_M:\til{M}\to M$, $\pi_{L_1}:L_1\times_M\til{M}_2\to L_1$, $\pi_{L_2}:\til{M}_1\times_ML_2\to L_2$ are all covering maps, the tangent bundles are identified with the pull-backs:
$$T\til{M}\cong\pi_M^*TM,\quad T(L_1\times_M\til{M}_2)\cong\pi_{L_1}^*TL_1,\quad T(\til{M}_1\times_ML_2)\cong\pi_{L_2}^*TL_2.$$
Hence we have the identification between the domain (resp. image) of $D_{\til{u}}\dbar_{\til{J}_t}$ and the domain (resp. image) of $D_u\dbar_{J_t}$ as Banach spaces. The two linearized Cauchy-Riemann operators are then identified and $\cu{M}(\til{p},\til{q};\til{\beta})$ inherits a natural Kuranishi structure so that $(\pi_M)_*$ is an orientation preserving isomorphism.
\end{proof}

\begin{proposition}\label{thm:chain_isom}
Let $\mathbb{L}_1,\mathbb{L}_2$ be Lagrangian immersions of $(M,\omega)$. The projection map $\pi_M:\til{M}\to M$ induces a canonical isomorphism between Floer complexes
$$(CF(\til{L}_1,\til{L}_2),m_1)\cong(CF(\mathbb{L}_1,\mathbb{L}_2),m_1).$$
\end{proposition}
\begin{proof}
By Part b) of Lemma \ref{lem:intersection_points_correspondence}, the projection $\pi_M:\til{M}\to M$ gives an identification between the Floer complexes:
$$\pi_M:CF(\til{L}_1,\til{L}_2)\to CF(\mathbb{L}_1,\mathbb{L}_2).$$
It suffices to prove that $\pi_M$ is a chain map. By Lemma \ref{lem:intersection_points_correspondence}, for each $p\in\xi_1(L_1)\cap\xi_2(L_2)$, there exists a unique $\widetilde{p}\in\widetilde{L}_1\cap\widetilde{L}_2$ such that $\pi_M(\widetilde{p})=p$. Furthermore, for each $q\in\xi_1(L_1)\cap\xi_2(L_2)$ and $u\in\mathcal{M}^0(p,q;\beta)\neq\phi$, there exists a unique $\widetilde{q}\in\widetilde{L}_1\cap\widetilde{L}_2$ with $\pi_M(\widetilde{q})=q$ and a unique $\widetilde{u}\in\mathcal{M}^0(\widetilde{p},\widetilde{q};\til{\beta})$ such that $\pi_M\circ\widetilde{u}=u$. Hence we have
\begin{align*}
m_1(\pi_M(\widetilde{p})) & = \sum_{q\in\xi_1(L_1)\cap\xi_2(L_2)}\sum_{\beta:\mu(\beta)=1}\sum_{u\in\mathcal{M}^0(p,q,\beta)}(-1)^{\text{sign}(u)}T^{\omega(u)}\cdot q\\
& = \sum_{\widetilde{q}\in\widetilde{L}_1\cap\widetilde{L}_2}\sum_{\til{\beta}:\mu(\til{\beta})=1}\sum_{\widetilde{u}\in\mathcal{M}^0(\widetilde{q},\widetilde{p},\til{\beta})}(-1)^{\text{sign}(\widetilde{u})}T^{\widetilde{\omega}(\widetilde{u})}\cdot \pi_M(\widetilde{q}).
\end{align*}
The last summation is exactly $\pi_M(m_1(\widetilde{p}))$. Note that $\text{sign}(u) = \text{sign}(\widetilde{u})$ here because $(\pi_M)_*$ is an isomorphism between oriented Kuranishi spaces by the previous lemma.
\end{proof}

\begin{theorem}\label{thm:inv_thm}
Let $\bb{L}_1,\bb{L}_2$ be Lagrangian immersions in $(M,\omega)$. The Floer cohomology $HF(\bb{L}_1,\bb{L}_2)$ is invariant under lifted Hamiltonian isotopy. That is, if $\bb{L}_2$ is lifted Hamiltonian isotopic to $\bb{L}_2'$, then there is a quasi-isomorphism
$$(CF(\mathbb{L}_1,\mathbb{L}_2),m_1)\simeq(CF(\mathbb{L}_1,\mathbb{L}_2'),m_1).$$
\end{theorem}
\begin{proof}
It suffices to prove the theorem in the case when $\bb{L}_2$ is $(\til{M}_2,\pi_2)$-lifted Hamiltonian isotopic to $\bb{L}_2'$ for some finite unramified covering $\pi_2:\til{M}_2\to M$. Suppose $\bb{L}_2$ is $(\til{M}_2,\pi_2)$-lifted Hamiltonian isotopic to $\bb{L}_2'$. By definition, there exist liftings $\til{\xi}_2:L_2\to\til{M}_2$ and $\til{\xi}_2':L_2\to\til{M}_2$ such that $(L_2,\til{\xi}_2)$ is globally Hamiltonian equivalent to $(L_2,\til{\xi}_2')$. In this case, $\til{M}=M\times_M\til{M}_2\cong\til{M}_2$, so we have a quasi-isomorphism
$$(CF(\til{L}_1,\til{L}_2),m_1)\simeq(CF(\til{L}_1,\til{L}_2'),m_1).$$
Together with the isomorphism obtained in Proposition \ref{thm:chain_isom}, we have the quasi-isomorphism
$$(CF(\mathbb{L}_1,\mathbb{L}_2),m_1)\simeq(CF(\mathbb{L}_1,\mathbb{L}_2'),m_1).$$
This completes the proof.
\end{proof}

Theorem \ref{thm:inv_thm} shows that lifted Hamiltonian equivalence defines an equivalence relation on objects of the immersed Fukaya category of $(M,\omega)$, thus giving an answer to \cite[Question 13.15]{AJ} which asked for a restricted class of local Hamiltonian equivalences under which the immersed Lagrangian Floer cohomology is invariant.

In the context of mirror symmetry, one needs to complexify the Fukaya category by unitary local systems on the domain of the Lagrangian immersion. In this case, the differential $m_1$ on $CF((\mathbb{L}_1,\mathcal{L}_1),(\mathbb{L}_2,\mathcal{L}_2))$ should be coupled with the holonomy coming from the local systems $\mathcal{L}_1,\mathcal{L}_2$ on the boundary of the disks. The notion of lifted Hamiltonian isotopy can be generalized as follows

\begin{definition}
Let $\mathbb{L}_1 = (L_1, \xi_1), \mathbb{L}_2 = (L_2, \xi_2)$ be two Lagrangian immersions of $M$ and $\mathcal{L}_1, \mathcal{L}_2$ be local systems on $L_1, L_2$ respectively. Let $\pi:\til{M}\to M$ be a finite unramified covering of $M$. We say $(\mathbb{L}_1,\mathcal{L}_1)$ is {\em $(\til{M},\pi)$-lifted Hamiltonian isotopic} to $(\mathbb{L}_2,\mathcal{L}_2)$ if
\begin{itemize}
\item[(a)]
There exist a diffeomorphism $\phi:L_1\to L_2$ and liftings $\til{\xi}_1:L_1\to\til{M}$, $\til{\xi}_2:L_2\to\til{M}$ such that
$(L_1,\til{\xi}_1)$ is globally Hamiltonian isotopic to $(L_1,\til{\xi}_2\circ\phi)$ and
\item[(b)]
$\phi^*\cu{L}_2\cong\cu{L}_1$ as unitary bundles.
\end{itemize}
\end{definition}

With a slight modification, one can also prove the invariance of the Floer cohomology $HF((\mathbb{L}_1,\mathcal{L}_1),(\mathbb{L}_2,\mathcal{L}_2))$ under this generalized notion of lifted Hamiltonian isotopy for any pair of immersed Lagrangian branes $(\mathbb{L}_1,\mathcal{L}_1),(\mathbb{L}_2,\mathcal{L}_2)$. We omit the detailed proof.

\section{Mirror of isomorphism between holomorphic vector bundles}\label{sec:mirror_analog_isom}

Let us go back to the mirror symmetry between $X$ and $\check{X}$. In this section, we will prove, at least in the semi-flat and caustics-free case, that certain lifted Hamiltonian equivalence between immersed Lagrangian multi-sections of the fibration $X\to B$ is mirror to isomorphism between holomorphic vector bundles over the mirror $\check{X}$.

Let $\bb{L}=(L,\xi,c_r)$ be an immersed Lagrangian multi-section of $X\to B$. 
Recall that $c_r:L\to B$ is a finite unramified covering, and the projection $\pi_X:L\times_BX\to X$ is also a finite unramified covering of $X$. A deck transformation $\tau_L\in \text{Deck}(L/B)$ induces a deck transformation $\tau\in \text{Deck}(L\times_BX/X)$ by
$$\tau:(l,x,y)\mapsto(\tau_L(l),x,y).$$
Hence we get an injective group homomorphism $\text{Deck}(L/B) \to \text{Deck}(L \times_B X / X)$. Let $G$ be the image of this homomorphism. With respect to the pullback symplectic structure on $L\times_BX$, elements in $G$ are symplectomorphisms.

On the mirror side, we also have a finite unramified covering $\pi_{\check{X}}:L\times_B\check{X}\to\check{X}$. One can apply a similar construction to obtain an embedding $\text{Deck}(L/B)\hookrightarrow \text{Deck}(L \times_B \check{X} / \check{X})$. Denote the image by $\check{G}$. With respect to the pullback complex structure, $\check{G}$ is a subgroup of the group of biholomorphisms of $L\times_B\check{X}$. There is a natural bijection between $G$ and $\check{G}$ given by $G \cong \text{Deck}(L/B) \cong \check{G}$.

\begin{lemma}
If $L$ is connected and $\text{Deck}(L/B)$ acts transitively on fibers of $c_r:L\to B$, then $G = \text{Deck}(L \times_B X / X)$ and $\check{G} = \text{Deck}(L \times_B \check{X} / \check{X})$.
\end{lemma}
\begin{proof}
Since $L$ is connected, $L\times_BX$ is also connected, and so $\text{Deck}(L\times_BX/X)$ acts freely on the fiber of $\pi_X:L\times_BX\to X$. Hence $G$ also acts freely on fibers of $\pi_X$. Fix $(x,y)\in X$. The fiber of $\pi_X$ over $(x,y)$ is in bijection with the fiber of $c_r$ over $x\in B$. By assumption, $\text{Deck}(L/B)$ acts transitively on the fiber of $c_r$. Hence $G$ also acts transitively on the fiber of $\pi_X$. Therefore, $\text{Deck}(L\times_BX/X)$ also acts transitively on fibers of $\pi_X$. Since both $G$ and $\text{Deck}(L\times_BX/X)$ act transitively and freely on fibers, we must have $G = \text{Deck}(L\times_BX/X)$.
\end{proof}

\begin{remark}
The transitivity of the action of $\text{Deck}(L/B)$ on fibers of $c_r:L\to B$ is equivalent to the normality of $(c_r)_*(\pi_1(L))$ as a subgroup of $\pi_1(B)$.
\end{remark}

It is known by \cite{sections_line_bundles} that when $B$ is compact, the Lagrangian sections $L_1, L_2$ are (globally) Hamiltonian equivalent if and only if their SYZ mirrors $\check{\bb{L}}_1, \check{\bb{L}}_2$ are isomorphic as holomorphic line bundles. In the higher rank situation, the following theorem shows, at least with a transitivity assumption, that $(L\times_BX,\pi_X)$-lifted Hamiltonian equivalence is the mirror analog of isomorphism between holomorphic vector bundles.

\begin{theorem}\label{thm:isom_SYZ_mirror}
Suppose that $B$ is compact. Let $\mathbb{L}_1 = (L, \xi_1, c_r)$, $\mathbb{L}_2 = (L, \xi_2, c_r)$ be immersed Lagrangian multi-sections of $X\to B$ with the same connected domain $L$ and unramified covering map $c_r:L\to B$. Assume that $\text{Deck}(L/B)$ acts transitively on fibers of $c_r:L\to B$. Then $\mathbb{L}_1$ is $(L \times_B X, \pi_X)$-lifted Hamiltonian isotopic to $\mathbb{L}_2$ if and only if $\check{\mathbb{L}}_1$ is isomorphic to $\check{\mathbb{L}}_2$ as holomorphic vector bundles.
\end{theorem}
\begin{proof}
Given two immersed Lagrangian multi-sections $\bb{L}_1 = (L, \xi_1, c_r)$ and $\bb{L}_2 = (L, \xi_2, c_r)$ with the same connected domain $L$ and finite unramified covering map $c_r:L\to B$, we can lift $\bb{L}_j$, $j=1,2$ to Lagrangian embeddings $\til{\xi}_j:L\to\til{L}_j\subset L\times_BX$, explicitly given by
$$\til{\xi}_j:l\mapsto(l,x,\xi_j(l)).$$
Moreover, for any $\tau_L\in \text{Deck}(L/B)$ and section $\til{\xi}:L\to\til{L}\subset L\times_BX$ of the fibration $\pi_L:L\times_BX\to L$, the composition
$$\tau\circ\til{\xi}\circ\tau_L^{-1}:l\mapsto(l,x,\til{\xi}(\tau_L^{-1}(l)))$$
defines a section of $\pi_L$.

Suppose $\mathbb{L}_1$ is $(L\times_BX,\pi_X)$-lifted Hamiltonian isotopic to $\mathbb{L}_2$. Then there exist liftings $\til{\xi}_1',\til{\xi}_2':L\to L\times_BX$ such that $(L,\til{\xi}_1')$ and $(L,\til{\xi}_2')$ are globally Hamiltonian isotopic to each other. Since both $\til{\xi}_1,\til{\xi}_1'$ are liftings of $\xi_1$, by the transitivity assumption, there exists $\tau_1\in \text{Deck}(L\times_BX/X)$ such that
$$\til{\xi}_1'=\tau_1\circ\til{\xi}_1.$$
Similarly, there exists $\tau_2$ such that
$$\til{\xi}_2'=\tau_2\circ\til{\xi}_2.$$
In particular, $(L,\til{\xi}_1')$, $(L,\til{\xi}_2')$ are embedded Lagrangian submanifolds, and so $\tau_1\circ\til{\xi}_1\circ\tau_{L,1}^{-1}$ and $\tau_2\circ\til{\xi}_2\circ\tau_{L,2}^{-1}$ are globally Hamiltonian equivalent Lagrangian sections of $\pi_L$.

Let $\check{\bb{L}}_1$, $\check{\bb{L}}_2$ be the SYZ mirror line bundles of $(L,\til{\xi}_1)$ and $(L,\til{\xi}_2)$ respectively. Then the correspondence result of \cite{sections_line_bundles} gives an isomorphism $(\check{\tau}_1^{-1})^*\check{\bb{L}}_1\cong(\check{\tau}_2^{-1})^*\check{\bb{L}}_2$ as holomorphic line bundles, where $\check{\tau}_j\in\check{G}$ corresponds to $\tau_j\in G$ under the natural isomorphism $G\cong \text{Deck}(L/B)\cong\check{G}$. We have
$$\check{\mathbb{L}}_j(U)=((\pi_{\check{X}})_*\check{L_j})(U)=\check{L_j}(\pi_{\check{X}}^{-1}(U)),\quad j=1,2.$$
For $j=1,2$, we have
$$((\pi_{\check{X}})_*(\check{\tau}_j^{-1})^*\check{\bb{L}}_j)(U)=\check{\bb{L}}_j(\check{\tau}_j(\pi_{\check{X}}^{-1}(U)))=\check{\bb{L}}_j((\pi_{\check{X}}\circ\check{\tau}_j^{-1})^{-1}(U))=(\pi_{\check{X}})_*\check{\bb{L}}_j(U),$$
so that
$$\check{\mathbb{L}}_2=(\pi_{\check{X}})_*\check{\bb{L}}_2\cong(\pi_{\check{X}})_*(\check{\tau}_2^{-1})^*\check{\bb{L}}_2\cong(\pi_{\check{X}})_*(\check{\tau}_1^{-1})^*\check{\bb{L}}_1=(\pi_{\check{X}})_*\check{\bb{L}}_1=\check{\mathbb{L}}_1.$$

Conversely, suppose $\check{\mathbb{L}}_1\cong\check{\mathbb{L}}_2$. By the correspondence result of \cite{sections_line_bundles} again, it suffices to show that $\check{\bb{L}}_1\cong\check{\tau}^*\check{\bb{L}}_2$ for some $\check{\tau}\in\check{G}$. Note that since $L\times_BX$ is connected, we have the following decompositions:
$$\pi_{\check{X}}^*(\pi_{\check{X}})_*\check{\bb{L}}_j=\bigoplus_{\check{\tau}\in\check{G}}\check{\tau}^*\check{\bb{L}}_j,\quad j=1,2.$$
Hence
$$\bigoplus_{\check{\tau}\in\check{G}}\check{\tau}^*\check{\bb{L}}_1\cong\bigoplus_{\check{\tau}\in\check{G}}\check{\tau}^*\check{\bb{L}}_2.$$
In particular, $\check{\bb{L}}_1$ defines a subbundle of $\bigoplus_{\check{\tau}\in\check{G}}\check{\tau}^*\check{\bb{L}}_2$.

Since $\check{\bb{L}}_1$ is a subbundle, there exists $\check{\tau}_1\in\check{G}$ such that the composition
$$\check{\bb{L}}_1\hookrightarrow\bigoplus_{\check{\tau}\in\check{G}}\check{\tau}^*\check{\bb{L}}_2\to\check{\tau}_1^*\check{\bb{L}}_2$$
is not identically zero. Similarly, there exists $\check{\tau}_2\in\check{G}$ such that the composition $\check{\bb{L}}_2\hookrightarrow\bigoplus_{\check{\tau}\in\check{G}}\check{\tau}^*\check{\bb{L}}_1\to\check{\tau}_2^*\check{\bb{L}}_1$ is not identically zero. Therefore we obtain a chain of bundle maps
$$\check{\bb{L}}_1\to\check{\tau}_1^*\check{\bb{L}}_2\to\check{\tau}_1^*\check{\tau}_2^*\check{\bb{L}}_1\to\cdots\to(\check{\tau}_1^*\check{\tau}_2^*)^k\check{\bb{L}}_1,\quad k\geq 1,$$
each of which is not identically zero.

Take $k$ to be the order of $\check{\tau}_2\circ\check{\tau}_1$. Then we obtain a map $\check{\bb{L}}_1\to\check{\bb{L}}_1$, which is again, not identically zero. Since $B$ is compact, so is $\check{X}$. Hence $\check{\bb{L}}_1\to\check{\bb{L}}_1$ corresponds to a nonzero holomorphic function which can only be a nonzero constant. Therefore, $\check{\bb{L}}_1\to\check{\bb{L}}_1$ is an isomorphism and in particular, $\check{\bb{L}}_1\to\check{\tau}_1^*\check{\bb{L}}_2$ is injective. Since $\check{\bb{L}}_1$ and $\check{\tau}_1^*\check{\bb{L}}_2$ are line bundles, $\check{\bb{L}}_1\to\check{\tau}_1^*\check{\bb{L}}_2$ is an isomorphism.
\end{proof}

If we combine Theorem \ref{thm:isom_SYZ_mirror} with the surgery-extension correspondence theorem (Theorem \ref{thm:surgery_extension}), we obtain

\begin{corollary}\label{cor:torus}
Let $\bb{L}_1 = \bb{L}_{r_1,d_1}[c_1]$ and $\bb{L}_2 = \bb{L}_{r_2,d_2}[c_2]$. If $K, K'\subset \bb{L}_1\cap \bb{L}_2$ are sets of intersection points such that the Lagrangian surgeries $\bb{L}_K= \bb{L}_2\sharp_{K}\bb{L}_1$ and $\bb{L}_{K'}= \bb{L}_2\sharp_{K'}\bb{L}_1$ have connected domain and satisfy the gcd assumption $\text{gcd}(r_1 + r_2, d_1 + d_2) = 1$, then $\bb{L}_K$ and $\bb{L}_{K'}$ are $(S^1\times_{S^1}T^2,\pi_{T^2})$-lifted Hamiltonian isotopic to each other, and hence have isomorphic immersed Lagrangian Floer cohomologies.
\end{corollary}

We give an example to illustrate Corollary \ref{cor:torus}.
\begin{example}
Let
$$\bb{L}_1=\bb{L}_{1,0}[1/2],\quad \bb{L}_2=\bb{L}_{1,3}[0]$$
be Lagrangian straight lines in the standard symplectic torus $T^2$. Then $\bb{L}_1$ intersects $\bb{L}_2$ at three points, all of which are of index $1$. We equip $\bb{L}_1$ with the local system
$$d+2\pi i\frac{1}{2}dx$$
and $\bb{L}_2$ with the trivial one. Consider the Lagrangian immersions $\bb{L}_1,\bb{L}_3$, as shown in Figure \ref{fig:3}.

\begin{figure}[H]
\centering
   \includegraphics[width=70mm]{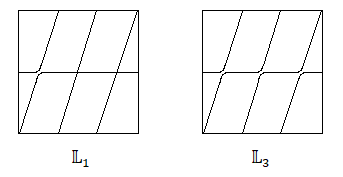}
   \caption{Two non-Hamiltonian equivalent but lifted Hamiltonian equivalent Lagrangian immersions in $T^2$.}
\label{fig:3}
\end{figure}

Both $\bb{L}_1,\bb{L}_3$ have connected domain and satisfy the gcd assumption: $\text{gcd}(1+1,0+3)=1$ as in Theorem \ref{thm:surgery_extension} If we equip their domain with the trivial local system, Theorem \ref{thm:surgery_extension} can be applied to conclude that both $\check{\bb{L}}_1$ and $\check{\bb{L}}_3$ fit into some exact sequences:
\begin{align*}
0 & \to \check{\bb{L}}_2 \to \check{\bb{L}}_1 \to \check{\bb{L}}_{1,\frac{1}{2}} \to 0,\\
0 & \to \check{\bb{L}}_2 \to \check{\bb{L}}_3 \to \check{\bb{L}}_{1,\frac{1}{2}} \to 0.
\end{align*}
By Atiyah's classification of indecomposable bundles on elliptic curves \cite{Atiyah_vector_bundle_over_an_elliptic_curve}, we know that $\check{\bb{L}}_1\cong\check{\bb{L}}_3$. Hence by Theorem \ref{thm:isom_SYZ_mirror}, $\bb{L}_1$ and $\bb{L}_3$ are $(S^1\times_{S^1}T^2,\pi_{T^2})$-lifted Hamiltonian isotopic to each other.

We can also compute the Floer cohomology of $\bb{L}_1$ and $\bb{L}_3$ directly. Since $\bb{L}_3$ is embedded and bounds no holomorphic disks, we have
$$HF(\bb{L}_3,\bb{L}_3)\cong H(S^1;\Lambda_{nov}).$$
For $\bb{L}_1$, let $\xi_1:S^1\to T^2$ be the immersion map. The Floer complex is given by
$$H(S^1;\Lambda_{nov})\oplus\Lambda_{nov}\{p^-,p^+,q^-,q^+\},$$
where $H(S^1;\Lambda_{nov})$ is the $\Lambda_{nov}$-valued cohomology of $S^1$ and $p^-,p^+,q^-,q^+$ are points on $S^1$ such that $\xi_1(p^-)=\xi_1(p^+)$ and $\xi_1(q^-)=\xi_1(q^+)$ are self-intersection points of $\bb{L}_1$ (see \cite[Corollary 11.4]{AJ}). Points with a positive (resp. negative) sign are graded to have degree 0 (resp. 1). There is one holomorphic disk from $p^+$ to $q^-$ and one from $q^+$ to $p^-$ (See Figure \ref{fig:holo_disk}).

\begin{figure}[H]
\centering
   \includegraphics[width=90mm]{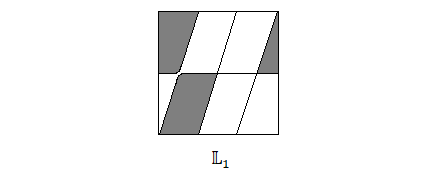}
   \caption{The holomorphic disk from $p^+$ (resp. $q^-$) to $q^+$ (resp. $p^-$).}
\label{fig:holo_disk}
\end{figure}

Hence the Floer cohomology of $\bb{L}_1$ is given by $HF(\bb{L}_1,\bb{L}_1)\cong H(S^1;\Lambda_{nov})$, which is canonically isomorphic to $HF(\bb{L}_3,\bb{L}_3)$, as expected by Theorem \ref{thm:inv_thm}. One can also use Hamiltonian perturbations to obtain the same result.
\end{example}

\appendix

\bibliographystyle{amsplain}
\bibliography{geometry}

\end{document}